\pdfoutput=1
\documentclass[11pt]{article}
\usepackage[margin=1.1in]{geometry}
\usepackage{amsmath,amssymb,amsthm,mathtools}
\usepackage{enumitem}
\usepackage[colorlinks=true,linkcolor=black,citecolor=black,urlcolor=black,filecolor=black]{hyperref}

\newtheorem{theorem}{Theorem}[section]
\newtheorem{lemma}[theorem]{Lemma}
\newtheorem{proposition}[theorem]{Proposition}
\newtheorem{corollary}[theorem]{Corollary}
\newtheorem{assumption}[theorem]{Assumption}
\theoremstyle{definition}
\newtheorem{definition}[theorem]{Definition}
\theoremstyle{remark}
\newtheorem{remark}[theorem]{Remark}

\DeclareMathOperator{\Tr}{Tr}
\DeclareMathOperator{\sinc}{sinc}
\newcommand{\R}{\mathbb{R}}

\newcommand{\F}{\mathcal{F}}
\newcommand{\eps}{\varepsilon}
\newcommand{\ind}[1]{\mathbf{1}_{#1}}
\newcommand{\HS}{\mathrm{HS}}
\newcommand{\logp}{\log_{2,+}}
\newcommand{\lnp}{\ln_{+}}
\newcommand{\Ltil}{\widetilde L}

\title{Tensor factorization and explicit spectral bounds\\
for product-box concentration operators}
\author{Ahmadreza Azimifard\thanks{Harmonic Research \& Technologies, LLC,
\href{mailto:afard@harmonicrt.com}{afard@harmonicrt.com}.}}
\date{July 2026}

\begin{document}
\maketitle

\begin{abstract}
Let $S=P_{cA_0}Q_{B_0}P_{cA_0}$ be the spatio-spectral concentration operator
of bounded sets $cA_0,B_0\subset\R^d$, and let
$\Lambda_\eps=\#\{n:\eps<\lambda_n(S)<1-\eps\}$ be its plunge count.

For $A_0$ and $B_0$ finite disjoint unions of bounded axis-parallel open
boxes we prove an explicit uniform upper bound on $\Lambda_\eps$: a single
estimate, with every constant written out in the side lengths, valid for
every $d\ge1$, every $c>0$ and every $0<\eps<\frac12$. On the range
$\alpha\ge4$, $c\ge2$, $\alpha^{-c}<\eps<\frac12$ it yields
$\Lambda_\eps\le C\,c^{d-1}\log\tfrac1\eps\,
\log\bigl(\alpha c/\log\tfrac1\eps\bigr)$. Kulikov and Dam Larsen
\cite[Thm.~1.3]{KDL} already prove that order on exactly that range, for a
class containing the pairs treated here. What the box class gains is the
all-parameter estimate itself, together with an independent proof.

The proof is organized around a telescoping tensorization of the
off-diagonal factor. A telescoping decomposition of $\ind{(cA_0)^c}$
writes the off-diagonal factor $P_{(cA_0)^c}Q_{B_0}P_{cA_0}$ as a sum of $d$
elementary tensor operators, each pairing \emph{one} copy of the
one-dimensional off-diagonal operator with $d-1$ localization factors.
Schatten quasi-norms multiply across tensor factors, and the single
logarithm comes only from the normal factor.

On the lower side, for the model cube pair, the plunge population is shown to
contain a $d$-fold product block once $\eps<4^{-d}$:
$\Lambda_\eps\ge M_a^{\,d}=\Omega\bigl((\log c)^d\bigr)$.
The ingredients are an exact trace identity, an explicit degree-three
polynomial minorant, and the sine-kernel determinant asymptotics of Basor and
Widom, quoted here from the literature; nothing is assumed about the
eigenvalue profile. The same machinery gives
$\Tr\bigl((S-S^2)^m\bigr)=\beta_m\pi^{-2}\log c+O_m(1)$ at every fixed order,
$\beta_m=B(m,m)$, together with a two-sided fixed-depth window estimate of
order $\log c$ that counts a tail window and its reflected head window.
Neither of these is uniform over its family, and the lower block is not a
matching bound. No statement of the paper is conditional on an unproved
hypothesis.
\end{abstract}

\section{Introduction}\label{sec:intro}

\subsection{The conjecture and the program}\label{sec:conj}

Fix the unitary Fourier transform
$\F f(\xi)=\int_{\R^d}f(x)e^{-2\pi ix\cdot\xi}\,dx$ on $L^2(\R^d)$. For
measurable $E\subset\R^d$ let $P_E$ be multiplication by $\ind E$ and
$Q_E=\F^{-1}P_E\F$. For bounded measurable $A_0,B_0\subset\R^d$ of positive
measure and a dilation parameter $c>0$, set $A=cA_0$, $B=B_0$ and
\[
S=S_{cA_0,B_0}=P_AQ_BP_A,\qquad
\Lambda_\eps=\Lambda_\eps(cA_0,B_0)=\#\{n:\ \eps<\lambda_n(S)<1-\eps\},
\quad 0<\eps<\tfrac12 .
\]
The precise counting convention, and its reconciliation with the half-open
convention used by Kulikov and Dam Larsen, are fixed once in
Definition~\ref{def:plunge} and Remark~\ref{rem:halfopen} below.
$S$ is a positive compact contraction (Lemma~\ref{lem:compact}); its
eigenvalues cluster at $1$ (about $|A||B|$ of them, the area law) and at
$0$, with a transition, the \emph{plunge region}, whose width is
governed by the boundary $\partial A$. Kulikov and Dam Larsen \cite{KDL}
proved the uniform bound
$\Lambda_\eps\lesssim c^{d-1}LR$, $L=\log(1/\eps)$, $R=\log(\alpha c/L)$,
when one of the sets is a finite union of parallelepipeds with disjoint
interiors and the other is bounded with finite upper Minkowski boundary
content \cite[Thm.~1.3]{KDL}, and $\Lambda_\eps\lesssim c^{d-1}LR^2$ for
general sets of finite upper Minkowski boundary content
\cite[Thm.~1.6]{KDL}; they conjectured (see the
discussion following \cite[Thm.~1.6]{KDL}, where they write that they believe
the extra logarithm ``does not actually take place'' but cannot remove it)
that the first form holds in general:
\begin{equation}\label{eq:conj}
\Lambda_\eps(cA_0,B_0)\ \le\ C(d,A_0,B_0,\alpha)\;c^{\,d-1}\,\log\frac1\eps\,
\log\Bigl(\frac{\alpha c}{\log(1/\eps)}\Bigr),
\qquad c\ge2,\ \alpha^{-c}<\eps<\tfrac12 .
\end{equation}
In the first paper of this program \cite{Afard1d} the first part was
carried out: an independent proof of \eqref{eq:conj} in $d=1$ (where
$c^{d-1}=1$ and the hypotheses force $A_0,B_0$ to be finite unions of
intervals), with explicit constants and with no restriction on $(c,\eps)$
beyond $\eps<\tfrac12$. The mechanism there was an exact
\emph{oscillation factorization} of the off-diagonal operator
$T=P_{A^c}Q_BP_A$ in boundary-distance coordinates, a scale-uniform
Hankel--Chebyshev estimate, and a boundary layer of width
$\asymp\log(1/\eps)$, assembled by Rotfel'd $p$-quasi-norm subadditivity
across the dyadic scales of a one-variable decomposition.

The present paper is the next step of the program: the \emph{product case}.
A single all-parameter estimate establishes \eqref{eq:conj} for finite
disjoint unions of axis-parallel boxes in every dimension. For the model cube
pair we prove, unconditionally, a complementary lower bound: a genuine
$d$-fold tensor block of plunge eigenvalues, of size $\Omega((\log c)^d)$,
together with an unconditional trace estimate in the same geometry. The lower
bound certifies that the plunge population of a box pair is not a
one-dimensional phenomenon lifted trivially. It does not match the upper
bound's order.

On the upper side, \eqref{eq:conj} in the box class is not new. A finite
disjoint union of bounded axis-parallel boxes is bounded, has boundary of
finite upper Minkowski content, and is a finite union of parallelepipeds with
disjoint interiors; so a pair $(A_0,B_0)$ satisfying
Assumption~\ref{ass:boxes} satisfies the hypotheses of
\cite[Thm.~1.3]{KDL}, whose conclusion is \eqref{eq:conj} on exactly the
range stated there. What Theorem~\ref{thm:upper} supplies that
\cite[Thm.~1.3]{KDL} does not is an estimate valid at \emph{every} $c>0$ and
\emph{every} $\eps\in(0,\frac12)$. Its constants are written out in the side
lengths rather than left implicit, and its proof is independent.
Section~\ref{sec:priorwork} compares the two statement by statement.

\begin{assumption}\label{ass:boxes}
$A_0=\bigsqcup_{i=1}^{M}R^{(i)}$ and $B_0=\bigsqcup_{j=1}^{K}Q^{(j)}$ are
finite disjoint unions of bounded axis-parallel open boxes,
\[
R^{(i)}=\prod_{m=1}^d I^{(i)}_m,\qquad
Q^{(j)}=\prod_{m=1}^d B^{(j)}_m,
\]
with $|I^{(i)}_m|=a^{(i)}_m>0$ and $|B^{(j)}_m|=b^{(j)}_m>0$.
\end{assumption}

\subsection{Main results}\label{sec:main}

Throughout, $\lnp x=\max(0,\ln x)$ and $\logp x=\max(0,\log_2x)$, and for
$0<\eps<\frac12$,
\begin{equation}\label{eq:Ltil}
\Ltil=\ln\frac{1}{\eps(1-\eps)}\in(\ln4,\infty),\qquad
L=\ln\frac1\eps,\qquad
p=\frac1\Ltil\in(0,1).
\end{equation}
The two logarithmic parameters $L$ and $\Ltil$ are distinct and both
load-bearing; they are compared in \eqref{eq:Lchain}, and the logarithmic
factor in $c$ is fixed in Section~\ref{sec:logparams}. For $\ell,b>0$ define
the \emph{normal profile function}
\begin{equation}\label{eq:G}
G(\ell,b)\;=\;\pi e\,b+12.5+4.5\,\logp\!\bigl(\ell/\Ltil\bigr),
\end{equation}
Despite the notation, $G$ depends on $\eps$ as well as on $(\ell,b)$,
through $\Ltil$; the two displayed arguments are the ones that vary from
coordinate to coordinate, whereas $\eps$ is fixed throughout any single
application. The same is true of $H$ below.
Abbreviate, for a component pair $(i,j)$ and a coordinate $m$,
\begin{equation}\label{eq:GH}
\ell^{(i)}_m=c\,a^{(i)}_m,\qquad
G^{(ij)}_m=G\bigl(\ell^{(i)}_m,\,b^{(j)}_m\bigr),\qquad
H^{(ij)}_m=2\,\ell^{(i)}_m b^{(j)}_m+2\sqrt2\,\Ltil\,G^{(ij)}_m .
\end{equation}

The upper bound comes first, in the uniform form from which everything else
is deduced.

\begin{theorem}[Product case, uniform form]\label{thm:upper}
Under Assumption~\ref{ass:boxes}, for every $d\ge1$, every $c>0$ and every
$0<\eps<\frac12$,
\[
\Lambda_\eps(cA_0,B_0)\;\le\;
2e^{1/2}\sum_{i=1}^{M}\sum_{j=1}^{K}\sum_{k=1}^{d}
\Ltil\,G^{(ij)}_k\prod_{m\neq k}H^{(ij)}_m .
\]
\end{theorem}

\begin{corollary}[an explicit KDL-form bound for product-box
unions]\label{cor:kdl}
Under Assumption~\ref{ass:boxes}, for every $\alpha\ge4$, $c\ge2$ and
$\alpha^{-c}<\eps<\frac12$,
\[
\Lambda_\eps(cA_0,B_0)\;\le\;C(d,A_0,B_0,\alpha)\;
c^{\,d-1}\,\ln\frac1\eps\,\ln\Bigl(\frac{\alpha c}{\ln(1/\eps)}\Bigr),
\]
with the explicit constant
\[
C=\sum_{i,j}\sum_{k=1}^d\gamma^{(ij)}_k\prod_{m\neq k}\kappa^{(ij)}_m,
\quad
\begin{aligned}
\kappa^{(ij)}_m&=2a^{(i)}_mb^{(j)}_m
+6\ln\alpha\,\bigl(\pi e\,b^{(j)}_m+12.5\bigr)+7\,a^{(i)}_m,\\
\gamma^{(ij)}_k&=6.3\,\bigl(\pi e\,b^{(j)}_k+13\bigr)
+41\,\bigl(1+\lnp a^{(i)}_k\bigr).
\end{aligned}
\]
\end{corollary}

\noindent
The order in Corollary~\ref{cor:kdl} is not new, and its proof here is
independent. Under Assumption~\ref{ass:boxes} the pair $(A_0,B_0)$ falls
inside the geometric class of \cite[Thm.~1.3]{KDL}, in which $A$ may be any
bounded set with boundary of finite upper Minkowski content while $B$ is a
finite union of parallelepipeds with disjoint interiors; and
\cite[Thm.~1.3]{KDL} delivers the same $c^{\,d-1}LR$ order on the same range
$c\ge2$, $\alpha^{-c}<\eps<\frac12$. Corollary~\ref{cor:kdl} is therefore a
different derivation of a bound already available, obtained from the
product-box tensor factorization of Section~\ref{sec:mech} and carrying a
constant $C$ written out in the side lengths. It is stated here because
Theorem~\ref{thm:upper}, from which it follows in three lines, is
not available from \cite{KDL}: that theorem holds at every $c>0$
and every $\eps\in(0,\frac12)$, with no threshold in either variable and no
existentially quantified $\alpha(d,A_0,B_0)$.

Theorem~\ref{thm:upper} is uniform over \emph{all} $\eps\in(0,\frac12)$ and all
$c>0$; for $\eps$ below $\alpha^{-c}$ it degrades gracefully (each factor
$H_m$ becomes $\asymp\Ltil$, each $G_k$ becomes $O(1+b)$, so
$\Lambda_\eps\lesssim\Ltil^{\,d}$, the natural product-regime shape for the
super-exponentially small eigenvalues of a $d$-dimensional product
operator). That is the correct \emph{shape} but it is not the last word
there. For $\eps\le\alpha^{-c}$, \cite[Thm.~1.3]{KDL} determines the
\emph{lower-half} count exactly,
$\Lambda^-_\eps\asymp\bigl(\log\frac1\eps\,/\log(\log\frac1\eps\,/\,c)
\bigr)^d$, which is smaller than $\Ltil^{\,d}$ by the factor
$\bigl(\log(\log\frac1\eps\,/\,c)\bigr)^d$; and since no eigenvalue exceeds
$1-\eps$ in that regime, they thereby determine the whole count as well, as
$\Lambda^-_\eps$ plus the $\asymp c^{\,d}$ eigenvalues above $\frac12$.
Our estimate is not an improvement on \cite{KDL} there.
For $d=1$ the empty products make Theorem~\ref{thm:upper} reduce to
the (sharper, intermediate) per-component form of the main theorem of
\cite{Afard1d}, and hence imply it.

The lower bounds are stated for the model box pair $A_0=B_0=(0,1)^d$ with
$c=\ell$, the case in which the operator is an exact $d$-fold tensor power
and the counting question reduces to counting products of the
one-dimensional eigenvalues $\lambda_n(\ell)$ of Section~\ref{sec:1d}. Write
\begin{equation}\label{eq:Ma}
M_a:=\#\{n:\lambda_n(\ell)\in(\tfrac14,\tfrac34)\}
\end{equation}
for the one-dimensional window count.

\begin{theorem}[An unconditional logarithmic tensor
block]\label{thm:lower}
Let $A_0=B_0=(0,1)^d$, $d\ge2$, and $c=\ell$. Then for every $\eps<4^{-d}$,
\[
\Lambda_\eps(c)\ \ge\ M_a^{\,d}\ \ge\
\bigl((c_0\ln\ell-C_0)_+\bigr)^d
\ =\ \Omega\!\bigl((\log c)^d\bigr),
\qquad c_0=\frac{8}{15\pi^2}=0.05403\ldots
\]
No hypothesis on the eigenvalue profile is used. The window count $M_a$ is
obtained in Section~\ref{sec:lower} from an explicit polynomial moment
inequality (Proposition~\ref{prop:count}) evaluated against the second and
third traces of $\varphi=S-S^2$. Of these the exact trace identity of
Proposition~\ref{prop:trace} supplies the first only; the two the bound
actually consumes come from Corollary~\ref{cor:windowexact} of
Appendix~\ref{app:exact}, and hence from the determinant asymptotics of
Theorem~\ref{thm:basorwidom}. The coefficient $c_0$ is exact and explicit: it
is the coefficient delivered by that displayed degree-three minorant, and it
is the largest available from any minorant of degree at most three
(Lemma~\ref{lem:degree3}). It is \emph{not} the all-degree ceiling, which is
$2\ln3/\pi^2$ (Remark~\ref{rem:c0status}). The remainder constant $C_0$ is
finite but not effective; see the note following Theorem~\ref{thm:window}.
\end{theorem}

\noindent
What Theorem~\ref{thm:lower} delivers is a genuine
\emph{tensor block}: $M_a^{\,d}$ distinct eigenvalue slots of the
$d$-dimensional operator, each an honest product of $d$ one-dimensional
plunge eigenvalues, all of them strictly inside $(\eps,1-\eps)$. The same
window count also yields, with no further input, an unconditional lower
bound on the Schatten mass $\sum_n\lambda_n^{\,q}$ for $0<q\le\frac12$
(Theorem~\ref{thm:nq}).

Theorem~\ref{thm:lower} is \emph{not} a matching lower bound. The upper bound
of Corollary~\ref{cor:kdl} is of order $c^{\,d-1}LR$ and
Theorem~\ref{thm:lower} of order $(\log c)^d$; the two differ by a power of
$c$. Closing that gap would require one-dimensional lower-half density
estimates in a depth range the fixed-window argument used here does not
reach. Estimates of that kind are known for cubes in a deep polynomial
regime: \cite[Thm.~1.1]{KDL} bounds $\Lambda^-_\eps$ from below by
$c^{\,d-1}\log\frac1\eps\log(\alpha_dc/\log\frac1\eps)$ for
$\alpha_d^{-c}<\eps<c^{-\alpha_d}$. What is missing is a bound of that shape
holding uniformly for $\eps$ moving between fixed depth and that regime;
obtaining one by the trace method used here is a separate problem
(Section~\ref{sec:stops}), and no statement of this paper depends on it.

\begin{remark}[the two pairs of logarithmic parameters]\label{rem:matched}
Two pairs of logarithmic parameters are in circulation for this problem, and
they are converted here rather than identified silently. The upper bound is
produced by its proof in $L=\ln(1/\eps)$ with the natural-logarithm factor
$X=\ln(\alpha c/L)$; lower-bound statements are customarily written in
$\Ltil$ with the base-two factor $\log_2(c/\Ltil)$.
By \eqref{eq:Lchain} the first pair differs by at most a factor $2$, whereas
the second differs \emph{additively} and admits no uniform multiplicative
comparison; Section~\ref{sec:logparams} carries out both conversions and
records why the second one fails. The next corollary spends the conversion at
$\alpha=4$, with an explicit factor, so that displays in the two parameter
sets can be read against each other directly.
\end{remark}

\begin{corollary}[matched form]\label{cor:matched}
Under Assumption~\ref{ass:boxes}, for every $c\ge2$ and every
$4^{-c}<\eps<\frac12$,
\[
\Lambda_\eps(cA_0,B_0)\;\le\;C'\;c^{\,d-1}\;\Ltil\;\log_2\frac{4c}{\Ltil},
\qquad C'=2.01\,C\bigr|_{\alpha=4},
\]
with $C$ as in Corollary~\ref{cor:kdl}.
\end{corollary}

\noindent
Corollary~\ref{cor:matched} is Corollary~\ref{cor:kdl} with $\alpha$ fixed at
$4$ and both logarithmic factors rewritten in the $(\Ltil,\log_2(c/\Ltil))$
parameters; it is derived from Corollary~\ref{cor:kdl} in
Section~\ref{sec:upper}, immediately after it. It is strictly weaker than
Corollary~\ref{cor:kdl}: the free parameter $\alpha$ is spent, narrowing the
range from $\alpha^{-c}<\eps$ to $4^{-c}<\eps$, and a factor $2.01$ is paid
for the conversion. Its purpose is to let the upper bound be read directly
against lower-bound statements written in those parameters. The logarithmic
factor is written out as $\log_2(4c/\Ltil)$ rather than abbreviated, so that
every constant in it is visible. Appendix~\ref{app:constants} comments on the
auxiliary constants of both corollaries.

\subsection{The mechanism, and the obstruction it answers}\label{sec:mech}

The known route to uniform plunge bounds is: reduce to the off-diagonal
factor $T=P_{A^c}Q_BP_A$ via $S-S^2=T^*T$; control $\|T\|_p^p$ at the
self-tuned exponent $p=1/\Ltil$ (so that the Schatten--Markov inequality
costs only a factor $e^{1/2}$); and assemble $\|T\|_p^p$ from pieces using
the only general inequality available in the quasi-Banach range $0<p<1$,
Rotfel'd subadditivity $\|\sum X_k\|_p^p\le\sum\|X_k\|_p^p$. The danger is
well known: subadditivity treats the pieces as if their singular subspaces
were independent, and \emph{independent pieces genuinely add Schatten
mass}. The rank-one counterexample
$T_l=E_{ll}$, $\sum_lT_l=I_{D+1}$, $\|\sum T_l\|_p^p=D+1$, shows that no
Cotlar--Stein-type almost-orthogonality hypothesis can prevent this.
Each extra independent family of pieces therefore
threatens an extra logarithmic (or worse) factor; this is how the
general-set bounds of \cite{KDL} ($LR^2$) and, in a related setting, of
Marceca--Romero--Speckbacher \cite{MRS} (a higher power of the logarithm)
lose sharpness.

The way out, in product geometry, is that the pieces of the natural
decomposition are \emph{not} independent: they share singular subspaces
through an exact tensor structure. The mechanism has seven steps:
\begin{enumerate}[label=(M\arabic*),itemsep=1pt,topsep=3pt]
\item reduce the plunge count to the off-diagonal factor
$T=P_{(cA_0)^c}Q_{B_0}P_{cA_0}$, through $S-S^2=T^*T$ and Markov;
\item reduce to single box pairs $(R,Q)$, at an additive cost over the $MK$
components (Proposition~\ref{prop:pairs});
\item telescope the complement of the product box $R$ into exactly $d$
terms, \eqref{eq:telescope};
\item recognize each term as an elementary tensor operator with \emph{one}
normal one-dimensional off-diagonal factor and $d-1$ tangential
one-dimensional localization factors, \eqref{eq:tensor};
\item multiply Schatten quasi-norms across the tensor factors
(Lemma~\ref{lem:tensor}), an identity rather than an estimate;
\item read off the tangential factors, which carry area-law mass
$\asymp\ell_mb_m\asymp c$ each and so produce the $c^{\,d-1}$ surface scale
(Lemma~\ref{lem:mass});
\item read off the normal factor, which supplies the single logarithm and
nothing more (Proposition~\ref{prop:1d}).
\end{enumerate}
In detail: after reducing to a single pair of boxes
$(R,Q)$, the complement of $R=\prod_m(0,\ell_m)$ telescopes,
\begin{equation}\label{eq:telescope}
\ind{R^c}(x)\;=\;1-\prod_{m=1}^d\ind{(0,\ell_m)}(x_m)
\;=\;\sum_{k=1}^d\Bigl[\prod_{m<k}\ind{(0,\ell_m)}(x_m)\Bigr]\,
\ind{(0,\ell_k)^c}(x_k),
\end{equation}
splitting $T$ into exactly $d$ pieces, the $k$-th of which is the
\emph{elementary tensor operator}
\begin{equation}\label{eq:tensor}
\Bigl[\bigotimes_{m<k}S_m\Bigr]\otimes\;\widetilde T_k\;\otimes
\Bigl[\bigotimes_{m>k}Q_{B_m}P_{(0,\ell_m)}\Bigr],
\end{equation}
with $S_m$ the one-dimensional localization operators of
\eqref{eq:Sell} and $\widetilde T_k=P_{(0,\ell_k)^c}Q_{B_k}P_{(0,\ell_k)}$
exactly the one-dimensional off-diagonal operator of \cite{Afard1d}.
Schatten quasi-norms multiply across tensor factors
(Lemma~\ref{lem:tensor}), so every tangential mode rides the \emph{same}
normal singular profile $\widetilde T_k$ and the quasi-norm sees that
sharing exactly. The tangential factors then carry area-law mass
$\asymp\ell_mb_m\asymp c$ (Lemma~\ref{lem:mass}) and the normal factor the
plunge mass $\|\widetilde T_k\|_p^p\le2\Ltil G_k$ of
Proposition~\ref{prop:1d}; multiplying $d-1$ area factors by one plunge
factor gives $c^{d-1}LR$. Additive (Rotfel'd) losses are incurred only over
the $d$ directions of \eqref{eq:telescope}, the $MK$ component pairs and,
inside the one-dimensional input, the $O(\log)$ normal scales: one
logarithm in total, as conjectured. Section~\ref{sec:reduce} carries this
out.

The same tensor structure drives the lower bounds, in the opposite
direction. If $d$ one-dimensional eigenvalues each lie in the window
$(\tfrac14,\tfrac34)$, their product lies in $(4^{-d},(3/4)^d)$ and hence
inside $(\eps,1-\eps)$ for $\eps<4^{-d}$; so a one-dimensional window count
$M_a$ produces a $d$-dimensional block of size $M_a^{\,d}$. What must be
supplied is therefore purely one-dimensional: that the sinc-kernel spectrum
really does place $\gtrsim\log\ell$ eigenvalues in a fixed window.
Section~\ref{sec:lower} proves this from the trace identity together with the
sine-kernel determinant asymptotics of Basor and Widom, in the uniform form
proved by Charlier.

This is the ``cross-scale singular-subspace sharing'' that the obstruction
described above demands, realized structurally rather than
through orthogonality estimates. Its limits are the subject of
Section~\ref{sec:limits}: \eqref{eq:telescope} and \eqref{eq:tensor} use
the product structure of $R$ \emph{and} of $Q$, which is exactly what fails
for curved boundaries.

\subsection{Relation to prior work}\label{sec:priorwork}

The one-dimensional counting problem originates with Slepian, Landau and
Pollak, whose prolate spheroidal analysis identified the transition region
and its logarithmic width; Landau and Widom \cite{LandauWidom} made this
precise, showing that for fixed $\eps\in(0,\frac12)$
\[
\#\{n:\eps<\lambda_n(c)<1-\eps\}
=\frac{2}{\pi^2}\log c\,\log\Bigl(\frac1\eps-1\Bigr)+o(\log c),
\qquad c\to\infty .
\]
Non-asymptotic one-dimensional results in the same direction are due to
Kulikov \cite{Kulikov23,Kulikov26}, who bounds how far the leading
eigenvalues stay from $1$ before the plunge, and to Karnik, Romberg and
Davenport \cite{KRD21}, who prove two non-asymptotic \emph{upper} bounds on
the transition-region count $\#\{n:\eps<\lambda_n<1-\eps\}$ with the
Landau--Widom leading constant $2/\pi^2$, together with non-asymptotic
bounds on the individual eigenvalues, from below for the indices ahead of
the time--bandwidth product and from above for those behind it. The
one-dimensional input used here is that of \cite{Afard1d}; no argument below
consumes \cite{Kulikov23,Kulikov26,KRD21}.

In higher dimensions the comparison with \cite{KDL} is best made statement by
statement. Kulikov and Dam Larsen prove the $c^{\,d-1}LR$ bound
\cite[Thm.~1.3]{KDL} when one of the two sets is a finite union of
parallelepipeds with disjoint interiors and the other is bounded with finite
upper Minkowski boundary content, a $c^{\,d-1}LR^{2}$ bound
\cite[Thm.~1.6]{KDL} for general sets of finite upper Minkowski boundary
content, and, for the cube pair, two-sided estimates
\cite[Thm.~1.1]{KDL}. The last of these include a \emph{lower} bound of the
same $c^{\,d-1}LR$ order on the deep polynomial range
$\alpha_d^{-c}<\eps<c^{-\alpha_d}$, and the exact order
$\bigl(\log\frac1\eps\,/\log(\log\frac1\eps\,/\,c)\bigr)^d$ for
$\Lambda^-_\eps$ when $\eps\le\alpha_d^{-c}$.
\cite[Prop.~1.2]{KDL} extends Theorem~1.1 to axis-parallel boxes of
unequal side lengths, and \cite[Rmk.~1.4]{KDL} observes that for finite
unions of axis-parallel boxes the disjointness of interiors may be arranged
by subdivision and so costs nothing.

Three consequences follow, and none of them is in this paper's favour on the
question of order.

\emph{The present paper is not broader geometrically than \cite{KDL}.}
Assumption~\ref{ass:boxes} requires boxes on \emph{both} sides, which
\cite[Thm.~1.3]{KDL} does not.

\emph{Nor is Corollary~\ref{cor:kdl} independent of it.} Every pair
$(A_0,B_0)$ admitted by Assumption~\ref{ass:boxes} is admitted by
\cite[Thm.~1.3]{KDL}, whose conclusion on $c\ge2$, $\alpha^{-c}<\eps<\frac12$
has the same form. As far as the order goes, Corollary~\ref{cor:kdl} is
therefore \emph{contained in} \cite[Thm.~1.3]{KDL}, not incomparable to it.

\emph{What is not contained in \cite{KDL} is Theorem~\ref{thm:upper} itself.}
It is a single closed-form estimate holding at every $c>0$ and every
$\eps\in(0,\frac12)$, with every constant a written function of the side
lengths. The statements of \cite{KDL} carry an implied constant and an
existentially quantified $\alpha=\alpha(d,A_0,B_0)\ge4$, and their range is
$c\ge2$ with the split at $\eps=\alpha^{-c}$. Within product-box geometry the
present result gives a concrete all-parameter estimate and an independent
tensor-factorization proof. No comparison of constants is intended: those of
\cite{KDL} are implicit, and none is extracted anywhere below. The mechanism
is what the rest of the program deforms, being designed to localize along
$\partial A$ and therefore to have a chance at the general conjecture (the
case where \cite[Thm.~1.6]{KDL} falls one logarithm short).

A second line of work, by Israel, Mayeli and Hughes, meets this one at two
places. Israel and Mayeli \cite[Thm.~1.1]{IM24} bound the plunge count when
the spatial set is a hypercube and the frequency set a coordinate-wise
symmetric convex body, by an orthonormal system of tensor-product wave
packets built from a local sine basis, with $r^{\,d-1}\log(r/\eps)^{5/2}$ in
the surface term. For the cube pair $Q=[0,1]^d$, $S(r)=[-r,r]^d$ their
\cite[Thm.~1.4]{IM24} instead uses the observation that the cube-pair
operator is a $d$-fold composition of one-dimensional limiting operators,
one in each variable, so that its eigenvalues are products of one-dimensional
eigenvalues, and feeds the one-dimensional bounds of \cite{KRD21} through
that product to reach
$C_d\max\{r^{\,d-1}\log r\log\frac1\eps,(\log r\log\frac1\eps)^d\}$, assuming
$r\ge2\pi$ in their Fourier normalization and $\eps\in(0,\frac12)$. That
normalization is not the one used here, and the conversion is spent rather
than left implicit: \cite{IM24} sets
$\F f(\xi)=\int_{\R^d}f(x)e^{-ix\cdot\xi}\,dx$ where
Section~\ref{sec:conj} sets $e^{-2\pi ix\cdot\xi}$, so the band $[-r,r]$ has
width $r/\pi$ in the frequency normalization used here, and
\[
c=\frac r\pi,\qquad\text{so that } r\ge2\pi \text{ reads } c\ge2 .
\]
Their estimate overlaps Corollary~\ref{cor:kdl} on the model cube pair. After
the conversion the two bounds have the same asymptotic order for fixed
$\eps$, and also at fixed polynomial depth $\eps=c^{-q}$ with $q>0$ fixed as
$c\to\infty$. They are not uniformly identical over the full moving-$\eps$
range: the factor $\log r$ of \cite[Thm.~1.4]{IM24} is fixed by $c$ alone,
whereas the factor $\log\bigl(\alpha c/\log\frac1\eps\bigr)$ of
Corollary~\ref{cor:kdl} tunes itself to $\eps$ and collapses to $O(1)$ once
$\log\frac1\eps$ approaches the exponential scale $\asymp c$, which is where
the two separate. Theorem~\ref{thm:upper} additionally covers finite unions
of boxes with arbitrary side lengths, all $c>0$, and constants written
explicitly in those side lengths in place of a dimensional $C_d$. Off the
box class, Hughes, Israel
and Mayeli \cite{HIM2} treat maximally Ahlfors-regular boundaries by a
two-stage dyadic decomposition into cubical components, improving
Marceca, Romero and Speckbacher \cite{MRS} on that class; \cite{HIMdisk}
treats the disk and the ellipse by a boundary-adapted wave-packet frame with
Gevrey-$s$ cutoffs, reaching logarithmic exponent $1+2s$ for each $s>1$; and
\cite{HIMrough} converts the count into a bound on $\Tr(S-S^2)$, which at
fixed $\eps$ gives $\eps^{-1}c^{\,d-1}\log c$ whenever both sets merely have
finite perimeter. All four are upper bounds carrying an implied constant,
and the last buys its single logarithm in $c$ by trading the factor $L$ for
a power of $\frac1\eps$.

Two-term Szeg\H{o}-type asymptotics for \emph{smooth} functions of
$S_{A,B}$, which do not control the sharp counting function, go back to Widom
and, for piecewise-smooth domains, Sobolev \cite{Sobolev}.

\subsection{Plan}\label{sec:plan}

Section~\ref{sec:prelim} fixes notation and quotes the one-dimensional
input. Sections~\ref{sec:reduce}--\ref{sec:upper} prove
Theorem~\ref{thm:upper}, Corollary~\ref{cor:kdl} and
Corollary~\ref{cor:matched}.
Section~\ref{sec:lower} proves the lower bounds: the window count
(Theorem~\ref{thm:window}), the higher traces at every fixed order
(Proposition~\ref{lem:higher}), the polynomial-minorant lemma
(Lemma~\ref{lem:minorant}) and the fixed-depth density it supports
(Proposition~\ref{prop:fixeddepth}), the tensor block
(Theorem~\ref{thm:lower}) and the trace estimate (Theorem~\ref{thm:nq}).
Section~\ref{sec:limits} states what breaks off the box class, and signposts
the remaining steps of the program: curved spatial boundaries, curved
frequency sets, the curved pair, and recursive skeletons.

The division that governs the whole paper is this.
Sections~\ref{sec:reduce}--\ref{sec:upper}, the entire upper bound, assume
nothing beyond one quoted external estimate, the one-dimensional
off-diagonal bound of \cite{Afard1d} recorded as Proposition~\ref{prop:1d};
everything built on it is proved here. Section~\ref{sec:lower} rests instead
on a single cited theorem, the sine-kernel determinant asymptotics of Basor
and Widom in the uniform form proved by Charlier
\cite{BasorWidom83,Charlier21}, reached through Appendix~\ref{app:exact}. Its
results (Theorem~\ref{thm:window}, Proposition~\ref{lem:higher},
Proposition~\ref{prop:fixeddepth}, Theorem~\ref{thm:lower} and
Theorem~\ref{thm:nq}) assume nothing about the eigenvalue profile. There
is no third tier: no statement of this paper is conditional on an unproved
hypothesis. Two limitations are stated as plainly as the results themselves:
there is no uniformity in the trace order $m$ or in the depth $u$
(Remark~\ref{rem:mquant}), and no matching lower bound.

\section{Preliminaries}\label{sec:prelim}

\subsection{The plunge count}\label{sec:plunge}

\begin{definition}[plunge count]\label{def:plunge}
For $0<\eps<\tfrac12$,
\[
\Lambda_\eps(cA_0,B_0)=\#\{n:\ \eps<\lambda_n(S)<1-\eps\},
\qquad S=P_{cA_0}Q_{B_0}P_{cA_0},
\]
the eigenvalues $\lambda_n(S)$ being listed in decreasing order with
multiplicity.
\end{definition}

\begin{remark}[half-open convention, and the KDL half-counts]\label{rem:halfopen}
KDL count plunge eigenvalues with the half-open convention
$\#\{n:\ \eps<\lambda_n(S)\le1-\eps\}$. Every bound below extends to that
count with the same constant, by applying it at $\eps'<\eps$ and letting
$\eps'\uparrow\eps$, the bound being continuous in $\eps$.
The two counts can differ, by the multiplicity of the eigenvalue
$1-\eps$ when there is one; it is the \emph{bounds}, not the counts, that
transfer. All statements in this paper use Definition~\ref{def:plunge}
unless the half-open count is named explicitly.

Since Section~\ref{sec:scopelower} and Section~\ref{sec:stops} refer to them,
we record the two half-counts of \cite[(1.1)]{KDL}, with their arrangement of
strict and non-strict inequalities:
\[
\Lambda^+_\eps=\#\Bigl\{n:\ 1-\eps\ge\lambda_n>\tfrac12\Bigr\},\qquad
\Lambda^-_\eps=\#\Bigl\{n:\ \tfrac12\ge\lambda_n>\eps\Bigr\},
\]
so that
$\Lambda^+_\eps+\Lambda^-_\eps=\#\{n:\ \eps<\lambda_n\le1-\eps\}$, the
half-open count itself. The \emph{lower-half} count $\Lambda^-_\eps$ collects
the eigenvalues strictly above the threshold and at or below the plunge
midpoint, and it is the one every lower-side statement below is compared
against. Nothing in this paper is sensitive to the endpoint conventions; they
are fixed once so that the comparisons are unambiguous.
\end{remark}

\subsection{The logarithmic parameters}\label{sec:logparams}

The two parameters $L$ and $\Ltil$ of \eqref{eq:Ltil} are distinct, and both
occur: the uniform Theorem~\ref{thm:upper} is stated through
$p=1/\Ltil$, while Corollary~\ref{cor:kdl} is stated in $L$. They are
related by
\begin{equation}\label{eq:Lchain}
L\ \le\ \Ltil\ \le\ L+\ln2\ \le\ 2L\ \le\ 2c\ln\alpha,
\end{equation}
where the first three inequalities hold throughout $0<\eps<\frac12$ and the
last uses the hypotheses of Corollary~\ref{cor:kdl}, namely $\alpha\ge4$,
$c\ge2$ and $\alpha^{-c}<\eps<\frac12$, which give
$L=\ln(1/\eps)\in(\ln2,\,c\ln\alpha)$. The chain \eqref{eq:Lchain} is used
in exactly that scope; outside it, only $L\le\Ltil\le L+\ln2\le2L$ is
available.

For the logarithmic factor in $c$ two forms are in circulation, and we
convert between them once here rather than silently identify them. The factor
produced by the proof of Corollary~\ref{cor:kdl} is the natural logarithm
\begin{equation}\label{eq:X}
X=\ln\frac{\alpha c}{L},
\end{equation}
which is what appears in the statement of that corollary. Lower-bound
statements in this problem are customarily written instead with the base-two
factor $\log_2(c/\Ltil)$, whose denominator is $\Ltil$ and not a constant. It
is therefore \emph{not} of order $\ln\ell$, and no statement in this paper
reads it that way. The two orders part company on the range where $\Ltil$ is
comparable to $c$: there $\log_2(c/\Ltil)=O(1)$ while $\ln\ell\to\infty$.

$X$ and $\log_2(c/\Ltil)$ are nevertheless comparable, up to an additive
constant. Since $\log_2x=\ln x/\ln2$ and, by \eqref{eq:Lchain},
$L\le\Ltil\le2L$,
\begin{equation}\label{eq:Rconvert}
\ln\frac{\alpha c}{\Ltil}\ \le\ X\ \le\ \ln\frac{\alpha c}{\Ltil}+\ln2 ,
\end{equation}
and $\ln\frac{\alpha c}{\Ltil}=\ln\alpha+\ln2\cdot\log_2\frac c{\Ltil}$, so
$X$ and $\ln2\cdot\log_2\frac c{\Ltil}$ differ by an additive constant
depending only on $\alpha$, lying between $\ln\alpha$ and $\ln\alpha+\ln2$
rather than by $\ln2$ alone. The comparison is additive and not
multiplicative: the
ratio $X/(\ln2\cdot\log_2\frac c{\Ltil})$ is unbounded as $\Ltil\uparrow c$,
since $X\ge\ln\alpha>0$ while $\ln(c/\Ltil)\to0$.
The multiplicative comparison is carried out once, at
$\alpha=4$, in the proof of Corollary~\ref{cor:matched}, with the explicit
factor $2.01$. No comparison uniform in $\alpha$ is asserted: that proof rests on
$\ln\frac{4c}\Ltil>\ln\frac2{\ln\alpha}$, which degrades as $\alpha$ grows
and is vacuous once $\alpha\ge e^2$.
Sections~\ref{sec:reduce}--\ref{sec:upper} use only $X$; it occurs in no
statement of Section~\ref{sec:lower}.

\subsection{Singular value calculus}\label{sec:svcalc}

For a compact operator $X$ between separable Hilbert spaces let
$s_1(X)\ge s_2(X)\ge\cdots$ be its singular values,
$n(t;X)=\#\{n:s_n(X)>t\}$ and $\|X\|_p^p=\sum_ns_n(X)^p$ for $0<p\le1$.

\begin{lemma}\label{lem:svcalc}
Let $X,Y$ be compact, $U,V$ bounded, $0<p\le1$, $t>0$. Then:
\emph{(i)} $s_n(UXV)\le\|U\|\|V\|s_n(X)$;
\emph{(ii)} \textup{(Rotfel'd \cite{Rotfeld}; cf.\ \cite[Thm.~2.8]{SimonTI})}
$\|X+Y\|_p^p\le\|X\|_p^p+\|Y\|_p^p$;
\emph{(iii)} \textup{(Markov)} $n(t;X)\le t^{-p}\|X\|_p^p$;
\emph{(iv)} if $U,V$ are multiplications by unimodular functions on the
output and input spaces, then $s_n(UXV)=s_n(X)$.
\end{lemma}

These are as in \cite[Lem.~2.1]{Afard1d}. We add the tensor identity, which
is the engine of this paper.

\begin{lemma}[Tensor multiplicativity]\label{lem:tensor}
Let $X_1,\dots,X_r$ be compact operators, $X_\nu:H_\nu\to H'_\nu$. Then the
singular values of $X_1\otimes\cdots\otimes X_r$ on
$H_1\otimes\cdots\otimes H_r$ are, as a multiset,
$\{\,s_{n_1}(X_1)\cdots s_{n_r}(X_r):\ n_1,\dots,n_r\ge1\,\}$, and
consequently, for every $0<p\le1$ (allowing $+\infty$),
\[
\|X_1\otimes\cdots\otimes X_r\|_p^p
=\prod_{\nu=1}^r\|X_\nu\|_p^p .
\]
\end{lemma}

\begin{proof}
By iteration it suffices to treat $r=2$. Since
$(X\otimes Y)^*(X\otimes Y)=(X^*X)\otimes(Y^*Y)$, it suffices to show that
for positive compact $A,B$ the spectrum (with multiplicity) of $A\otimes B$
is $\{\alpha_i\beta_j\}$ where $\{\alpha_i\},\{\beta_j\}$ are the
eigenvalues of $A,B$ repeated by multiplicity. Choose orthonormal bases
$\{u_i\}$ of $H_1$ and $\{v_j\}$ of $H_2$ consisting of eigenvectors
($Au_i=\alpha_iu_i$, $Bv_j=\beta_jv_j$; eigenvalue $0$ on the kernels).
Then $\{u_i\otimes v_j\}$ is an orthonormal basis of $H_1\otimes H_2$ and
$(A\otimes B)(u_i\otimes v_j)=\alpha_i\beta_j\,u_i\otimes v_j$, which is a
complete spectral resolution. The quasi-norm identity follows from
Tonelli's theorem:
$\sum_{i,j}(s_is'_j)^p=\bigl(\sum_is_i^p\bigr)\bigl(\sum_js'^{\,p}_j\bigr)$.
\end{proof}

\begin{remark}[the tensor structure, and where it is used]\label{rem:tensor}
For a single box pair $R=\prod_m(0,\ell_m)$, $Q=\prod_mB_m$, the operator
$S_{R,Q}=P_RQ_QP_R$ is the exact tensor product $\bigotimes_{m=1}^dS_m$ of
the one-dimensional operators $S_m=P_{(0,\ell_m)}Q_{B_m}P_{(0,\ell_m)}$, so
by Lemma~\ref{lem:tensor} its eigenvalues are exactly the products
$\prod_{m}\lambda_{n_m}(S_m)$. This identity is used in both directions:
upward through the Schatten quasi-norms in Section~\ref{sec:reduce}, and
downward, as a counting statement about products of one-dimensional
eigenvalues, in Section~\ref{sec:lower}.
\end{remark}

\subsection{Compactness and the plunge reduction}\label{sec:compact}

\begin{lemma}\label{lem:compact}
$S=P_AQ_BP_A$ is self-adjoint, $0\le S\le I$, and Hilbert--Schmidt.
Moreover, with $T=P_{A^c}Q_BP_A$,
\[
S-S^2=T^*T,\qquad\text{and}\qquad
\Lambda_\eps\le n\bigl(t;T\bigr),\quad t=\sqrt{\eps(1-\eps)},
\]
for every $0<\eps<\frac12$.
\end{lemma}

\begin{proof}
Identical to \cite[Lem.~2.4--2.5]{Afard1d}; the algebra
$T^*T=P_AQ_B(I-P_A)Q_BP_A=S-S^2$ and the spectral-mapping argument for
$\varphi(\lambda)=\lambda(1-\lambda)$, which exceeds $\eps(1-\eps)$ on
$(\eps,1-\eps)$, are dimension-free. Compactness: the kernel of $S$ is
$\ind A(x)g(x-y)\ind A(y)$ with $g=\F^{-1}\ind B\in L^2(\R^d)$, and
$\iint_{A\times A}|g(x-y)|^2\le|A|\,|B|<\infty$.
\end{proof}

\subsection{The one-dimensional operator and the one-dimensional
input}\label{sec:1d}

Both halves of the paper reduce to one and the same one-dimensional object,
which we fix once. Let $Q$ be the orthogonal projection of $L^2(\R)$ onto
the Paley--Wiener space of functions with Fourier support in
$(-\tfrac12,\tfrac12)$; its kernel is $q(x,y)=\sinc(x-y)$ with
$\sinc(t)=\sin(\pi t)/(\pi t)$, $\sinc(0)=1$, and $Q=Q^*=Q^2$. For
$I=(0,\ell)$ let $P=P_I$ and $P_{I^c}=1-P$, and set
\begin{equation}\label{eq:Sell}
S_\ell=P\,Q\,P\quad\text{on }L^2(I),\qquad
1>\lambda_0(\ell)\ge\lambda_1(\ell)\ge\cdots>0 .
\end{equation}
Here $c:=\ell b=\ell$ is the area parameter for the unit band $b=1$; a
general band width $b>0$ enters only through the product $\ell b$. We write
$\varphi:=S_\ell-S_\ell^2=S_\ell(1-S_\ell)$, so that $\varphi\succeq0$,
$\varphi=P\,Q\,P_{I^c}\,Q\,P$, and its eigenvalues are
$\varphi_n=\lambda_n(1-\lambda_n)\in[0,\tfrac14]$. Two elementary
projection identities are used repeatedly:
\begin{equation}\label{eq:proj}
\sinc(0)=1,\qquad \int_\R \sinc(x-y)^2\,dy=q(x,x)=1\quad(x\in\R),
\end{equation}
the second because $q(x,\cdot)\in\operatorname{ran}Q$ and $Q^2=Q$.

All one-dimensional information used by the upper bound enters through the
following proposition, which is \cite[Prop.~3.1, Prop.~5.1]{Afard1d}
specialized to a single pair of intervals (the factor $2$ accounts for the
two sides of the interval). We use it as a black box.

\begin{proposition}[one-dimensional off-diagonal mass;
\cite{Afard1d}]\label{prop:1d}
Let $I\subset\R$ be an interval of length $\ell$ and
$B^\circ=(-b/2,b/2)$. Then for every $0<p\le1$,
\[
\bigl\|P_{I^c}\,Q_{B^\circ}\,P_I\bigr\|_p^p
\;\le\;\frac{2}{p}\,\Bigl[\pi e\,b+12.5+4.5\,\logp(\ell p)\Bigr].
\]
\end{proposition}

By translation invariance ($\tau_v$ commutes with $Q_{B^\circ}$ and
conjugates $P_E$ to $P_{E+v}$) the bound is independent of the position of
$I$, and we will use it with $I=(0,\ell)$.

\begin{remark}[the model box pair]\label{rem:cube}
For $A_0=B_0=(0,1)^d$ the cube-pair operator is the exact $d$-fold tensor
power $S_\ell^{\otimes d}$ of \eqref{eq:Sell}, so by
Lemma~\ref{lem:tensor} its eigenvalues are the $d$-fold products
$\prod_{i=1}^d\lambda_{n_i}(\ell)$ and its Schatten quasi-norms multiply
across factors. Every lower bound in Section~\ref{sec:lower} is therefore a
statement about products of the numbers $\lambda_n(\ell)$, which is the
object controlled there.
\end{remark}

\section{Reduction to single box pairs and tensorization}\label{sec:reduce}

Fix $0<p\le1$ throughout this section.

\begin{proposition}[Pair reduction]\label{prop:pairs}
Under Assumption~\ref{ass:boxes}, with $A=cA_0$, $B=B_0$ and
$T=P_{A^c}Q_BP_A$,
\[
\|T\|_p^p\;\le\;\sum_{i=1}^M\sum_{j=1}^K
\bigl\|T^{(ij)}\bigr\|_p^p,\qquad
T^{(ij)}:=P_{(cR^{(i)})^c}\;Q_{Q^{(j)}-\beta^{(j)}}\;P_{cR^{(i)}},
\]
where $\beta^{(j)}$ is the centre of $Q^{(j)}$, so that
$Q^{(j)}-\beta^{(j)}=\prod_m(-b^{(j)}_m/2,\,b^{(j)}_m/2)$. Moreover each
$cR^{(i)}$ may be translated to $\prod_m(0,\ell^{(i)}_m)$ without changing
singular values.
\end{proposition}

\begin{proof}
As $\ind B=\sum_j\ind{Q^{(j)}}$ a.e., $Q_B=\sum_jQ_{Q^{(j)}}$, and as
$\ind A=\sum_i\ind{cR^{(i)}}$, $P_A=\sum_iP_{cR^{(i)}}$; Rotfel'd
(Lemma~\ref{lem:svcalc}(ii)) gives
$\|T\|_p^p\le\sum_{i,j}\|P_{A^c}Q_{Q^{(j)}}P_{cR^{(i)}}\|_p^p$. Since
$cR^{(i)}\subset A$ we have $A^c\subset(cR^{(i)})^c$, hence
$P_{A^c}=P_{A^c}P_{(cR^{(i)})^c}$ and, by Lemma~\ref{lem:svcalc}(i),
$s_n(P_{A^c}Q_{Q^{(j)}}P_{cR^{(i)}})\le
s_n(P_{(cR^{(i)})^c}Q_{Q^{(j)}}P_{cR^{(i)}})$. The modulation
$(M_\beta f)(x)=e^{2\pi i\beta\cdot x}f(x)$ is unitary, commutes with every
spatial projection, and conjugates $Q_E$ to $Q_{E-\beta}$; with
$\beta=\beta^{(j)}$ this centres the frequency box while leaving the
spatial projections untouched, so the singular values equal those of
$T^{(ij)}$. Translations are handled as noted after
Proposition~\ref{prop:1d}.
\end{proof}

\begin{proposition}[Telescoping tensorization]\label{prop:tensor}
Let
\[
R=\prod_{m=1}^d(0,\ell_m),\qquad
B^\circ=\prod_{m=1}^d(-b_m/2,b_m/2),\qquad
T^\circ=P_{R^c}Q_{B^\circ}P_R
\]
on $L^2(\R^d)$. Under the canonical identification of $L^2(\R^d)$ with
$\bigotimes_{m=1}^dL^2(\R)$,
\[
T^\circ=\sum_{k=1}^d\Theta_k,\qquad
\Theta_k=\Bigl[\bigotimes_{m<k}S_m\Bigr]\otimes\widetilde T_k\otimes
\Bigl[\bigotimes_{m>k}Q_{B^\circ_m}P_{(0,\ell_m)}\Bigr],
\]
where, in the $m$-th coordinate, $B^\circ_m=(-b_m/2,b_m/2)$,
\[
S_m=P_{(0,\ell_m)}Q_{B^\circ_m}P_{(0,\ell_m)},\qquad
\widetilde T_k=P_{(0,\ell_k)^c}Q_{B^\circ_k}P_{(0,\ell_k)} .
\]
Consequently, by Lemmas~\ref{lem:svcalc}(ii) and \ref{lem:tensor},
\begin{equation}\label{eq:assemble}
\|T^\circ\|_p^p\;\le\;\sum_{k=1}^d\;
\bigl\|\widetilde T_k\bigr\|_p^p\,
\prod_{m<k}\|S_m\|_p^p\,
\prod_{m>k}\bigl\|Q_{B^\circ_m}P_{(0,\ell_m)}\bigr\|_p^p .
\end{equation}
\end{proposition}

\begin{proof}
Under the identification, $\F_d=\bigotimes_m\F_1$, hence for the product
set $B^\circ$, $Q_{B^\circ}=\bigotimes_mQ_{B^\circ_m}$; likewise
$P_R=\bigotimes_mP_{(0,\ell_m)}$. Writing $u_m=\ind{(0,\ell_m)}(x_m)$, the
telescoping identity
$1-\prod_{m\le d}u_m=\sum_{k=1}^d\bigl(\prod_{m<k}u_m\bigr)(1-u_k)$
(immediate by induction: consecutive partial products differ by exactly the
$k$-th term) gives
\[
P_{R^c}=\sum_{k=1}^d\Pi_k,\qquad
\Pi_k=\Bigl[\bigotimes_{m<k}P_{(0,\ell_m)}\Bigr]\otimes
P_{(0,\ell_k)^c}\otimes\Bigl[\bigotimes_{m>k}I\Bigr].
\]
Since composition respects elementary tensors,
$\Theta_k=\Pi_kQ_{B^\circ}P_R$ has exactly the stated factors:
$P_{(0,\ell_m)}Q_{B^\circ_m}P_{(0,\ell_m)}=S_m$ for $m<k$;
$P_{(0,\ell_k)^c}Q_{B^\circ_k}P_{(0,\ell_k)}=\widetilde T_k$ in slot $k$;
and $I\cdot Q_{B^\circ_m}P_{(0,\ell_m)}$ for $m>k$. Each factor is
Hilbert--Schmidt (kernels $\ind{}\,g_m(x-y)\ind{}$ with
$g_m=\F^{-1}\ind{B^\circ_m}\in L^2$), hence compact, and each is
$p$-summable by Proposition~\ref{prop:1d} and Lemma~\ref{lem:mass} below, so
Lemma~\ref{lem:tensor} applies; \eqref{eq:assemble} follows by Rotfel'd
over the $d$ terms.
\end{proof}

\begin{remark}
The projections $\Pi_k$ are mutually orthogonal
($\Pi_k\Pi_{k'}=0$ for $k\ne k'$, because the $k$-th slots carry
$P_{(0,\ell_k)^c}P_{(0,\ell_k)}=0$ when $k<k'$), so \eqref{eq:telescope}
is in fact an orthogonal decomposition of $\ind{R^c}$; we only use
subadditivity.
\end{remark}

\section{\texorpdfstring{Tangential mass at
$q\downarrow0$}{Tangential mass as q goes to 0}}\label{sec:mass}

The factors in \eqref{eq:assemble} other than $\widetilde T_k$ are powers
of one-dimensional localization spectra:
$\|S_m\|_p^p=\sum_n\lambda_n(S_m)^p$ and, since
$(Q_{B^\circ_m}P_{(0,\ell_m)})^*(Q_{B^\circ_m}P_{(0,\ell_m)})
=P_{(0,\ell_m)}Q_{B^\circ_m}P_{(0,\ell_m)}=S_m$,
\[
\bigl\|Q_{B^\circ_m}P_{(0,\ell_m)}\bigr\|_p^p
=\sum_n\lambda_n(S_m)^{p/2}.
\]
As $0\le\lambda_n\le1$ and $p\ge p/2$, both are bounded by
$N_m:=\sum_n\lambda_n(S_m)^{p/2}$. What is needed is an estimate on $N_m$ of
a particular shape: a trace term plus a controlled plunge term, at exponents
tending to $0$. Lemma~\ref{lem:mass} supplies one, with explicit constants,
in two lines from Proposition~\ref{prop:1d} and Markov's inequality. It is
elementary, and no priority is claimed for it.

\begin{lemma}[Tangential Schatten mass]\label{lem:mass}
Let $I$ be an interval of length $\ell$, $B^\circ=(-b/2,b/2)$, and
$S=P_IQ_{B^\circ}P_I$. Then for every $0<q\le\frac12$,
\[
\sum_{n\ge1}\lambda_n(S)^{\,q}\;\le\;
2\,\ell b\;+\;\frac{\sqrt2}{q}\,
\Bigl[\pi e\,b+12.5+4.5\,\logp(2q\ell)\Bigr].
\]
\end{lemma}

\begin{proof}
Split at $\frac12$. For the head, Markov on the trace: since
$\Tr S=\|Q_{B^\circ}P_I\|_\HS^2
=\iint\ind I(y)|g(x-y)|^2\,dx\,dy=\ell\,\|g\|_{L^2}^2=\ell b$
(Plancherel, $g=\F^{-1}\ind{B^\circ}$),
\[
\#\{n:\lambda_n>\tfrac12\}\le2\Tr S=2\ell b,
\qquad\text{hence}\qquad
\sum_{\lambda_n>1/2}\lambda_n^q\le2\ell b .
\]
For the tail, use the off-diagonal factor
$\widetilde T=P_{I^c}Q_{B^\circ}P_I$: by Lemma~\ref{lem:compact} (in
$d=1$), $S-S^2=\widetilde T^*\widetilde T$, so the nonzero spectrum of
$S-S^2$ is $\{s_m(\widetilde T)^2\}$ with multiplicity, while by spectral
mapping it consists of the values $\lambda_n(1-\lambda_n)$. If
$0<\lambda_n\le\frac12$ then $\lambda_n\le2\lambda_n(1-\lambda_n)$, and the
values $\lambda_n(1-\lambda_n)$ over such $n$ form a sub-multiset of
$\{s_m(\widetilde T)^2\}$; therefore, for $0<q\le\frac12$ (so that
$2q\le1$),
\[
\sum_{\lambda_n\le1/2}\lambda_n^q
\;\le\;2^q\sum_{\lambda_n\le1/2}\bigl(\lambda_n(1-\lambda_n)\bigr)^q
\;\le\;\sqrt2\;\sum_m s_m(\widetilde T)^{2q}
=\sqrt2\,\bigl\|\widetilde T\bigr\|_{2q}^{2q}.
\]
Proposition~\ref{prop:1d} at exponent $2q\in(0,1]$ gives
$\|\widetilde T\|_{2q}^{2q}\le\frac{2}{2q}
[\pi e\,b+12.5+4.5\logp(2q\ell)]$. Combining the two parts proves the
lemma.
\end{proof}

\begin{remark}
The lemma is the quantitative form of the \emph{area law for quasi-norms}:
as $q\downarrow0$, $\lambda^q\to\ind{\lambda>0}$ pointwise, and the bound
says that at the self-tuned exponent the effective rank of a
one-dimensional localization factor is its time--bandwidth area $\ell b$ up
to a single plunge correction
$O\bigl(\Ltil\,(1+\logp(\ell/\Ltil))\bigr)$.
This rank-like (rather than log-like) behaviour of the tangential factors is
what converts quasi-norm \emph{multiplication} into the surface factor
$c^{d-1}$ instead of an extra logarithm.
\end{remark}

\section{\texorpdfstring{Upper bound: proof of
Theorem~\ref{thm:upper}}{Upper bound: proof of Theorem 1.2}}\label{sec:upper}

\begin{proof}[Proof of Theorem~\ref{thm:upper}]
Fix $0<\eps<\frac12$ and set $t=\sqrt{\eps(1-\eps)}$,
$\Ltil=\ln\frac1{\eps(1-\eps)}$, $p=1/\Ltil$. Since
$\eps(1-\eps)<\frac14$, $\Ltil>\ln4>1$ and $p\in(0,1)$; moreover
$t^{-p}=e^{1/2}$. By Lemma~\ref{lem:compact} and Markov
(Lemma~\ref{lem:svcalc}(iii)),
$\Lambda_\eps\le n(t;T)\le e^{1/2}\|T\|_p^p$. By
Propositions~\ref{prop:pairs} and \ref{prop:tensor},
\[
\|T\|_p^p\;\le\;\sum_{i,j}\sum_{k=1}^d
\bigl\|\widetilde T^{(ij)}_k\bigr\|_p^p\prod_{m\neq k}N^{(ij)}_m,
\qquad N^{(ij)}_m=\sum_n\lambda_n\bigl(S^{(ij)}_m\bigr)^{p/2},
\]
where in the pair $(i,j)$ the $m$-th coordinate data are
$(\ell^{(i)}_m,b^{(j)}_m)$. Proposition~\ref{prop:1d} gives
$\|\widetilde T^{(ij)}_k\|_p^p\le\frac2p\,G^{(ij)}_k
=2\Ltil\,G^{(ij)}_k$ (note $\logp(\ell p)=\logp(\ell/\Ltil)$, matching
\eqref{eq:G}). Lemma~\ref{lem:mass} with $q=p/2\le\frac12$ gives
\[
N^{(ij)}_m\le2\ell^{(i)}_mb^{(j)}_m
+\frac{\sqrt2}{p/2}\,G^{(ij)}_m
=2\ell^{(i)}_mb^{(j)}_m+2\sqrt2\,\Ltil\,G^{(ij)}_m=H^{(ij)}_m,
\]
again using $\logp(2\cdot\frac p2\,\ell)=\logp(\ell/\Ltil)$. Combining,
\[
\Lambda_\eps\le e^{1/2}\sum_{i,j}\sum_{k=1}^d
2\Ltil\,G^{(ij)}_k\prod_{m\neq k}H^{(ij)}_m,
\]
which is the assertion.
\end{proof}

\begin{proof}[Proof of Corollary~\ref{cor:kdl}]
Let $c\ge2$, $\alpha\ge4$ and $\alpha^{-c}<\eps<\frac12$; these are the
hypotheses under which \eqref{eq:Lchain} holds in full, so that
$L=\ln(1/\eps)\in(\ln2,\,c\ln\alpha)$ and
$L\le\Ltil\le L+\ln2\le2L\le2c\ln\alpha$. Throughout, drop the indices
$(i,j)$ and write $a=a^{(i)}_m$ etc.; also
$4.5\logp x=\frac{4.5}{\ln2}\lnp x\le6.5\lnp x$.

\emph{Each $H_m\le\kappa_m c$.} Indeed
$2\ell_mb_m=2a_mb_m\,c$;
$2\sqrt2\,\Ltil(\pi e\,b_m+12.5)\le4\sqrt2\,\ln\alpha\,
(\pi e\,b_m+12.5)\,c\le6\ln\alpha(\pi e\,b_m+12.5)\,c$;
and since $x\mapsto x\lnp(K/x)\le K/e$ for all $x>0$,
\[
2\sqrt2\cdot6.5\;\Ltil\,\lnp\!\frac{ca_m}{\Ltil}
\;\le\;18.4\cdot\frac{ca_m}{e}\;\le\;7\,a_m\,c .
\]

\emph{Each $2e^{1/2}\Ltil G_k\le\gamma_k\,L\ln(\alpha c/L)$.} Put
$X=\ln(\alpha c/L)$. Since $\alpha/\ln\alpha$ is increasing for
$\alpha\ge e$ and $L<c\ln\alpha$,
$X>\ln(\alpha/\ln\alpha)\ge\ln(4/\ln4)=:\eta=1.0596\ldots$ Next, since
$\Ltil\ge L$,
\[
\lnp\frac{ca_k}{\Ltil}\le\lnp\frac{ca_k}{L}
=\lnp\Bigl(\frac{\alpha c}{L}\cdot\frac{a_k}{\alpha}\Bigr)
\le X+\lnp a_k\le X\Bigl(1+\frac{\lnp a_k}{\eta}\Bigr),
\]
and $\pi e\,b_k+12.5\le\frac{X}{\eta}(\pi e\,b_k+12.5)$. Hence, with
$\Ltil\le2L$ and $2e^{1/2}\le3.3$,
\[
2e^{1/2}\,\Ltil\,G_k
\le6.6\,L\,X\Bigl[\frac{\pi e\,b_k+12.5}{\eta}
+6.5\Bigl(1+\frac{\lnp a_k}{\eta}\Bigr)\Bigr]
\le L\,X\,\Bigl[6.3(\pi e\,b_k+13)+41\bigl(1+\lnp a_k\bigr)\Bigr],
\]
i.e.\ $2e^{1/2}\Ltil G_k\le\gamma_k LX$. (Constants: $6.6/\eta\le6.3$
handles the $b_k$ term; for the remainder,
$6.6\cdot12.5/\eta+6.6\cdot6.5=120.8\le6.3\cdot13+41=122.9$ and
$6.6\cdot6.5/\eta=40.5\le41$.) Substituting both displays into
Theorem~\ref{thm:upper},
\[
\Lambda_\eps\le\sum_{i,j}\sum_k\gamma^{(ij)}_k\,
\Bigl[\prod_{m\neq k}\kappa^{(ij)}_m\Bigr]\,c^{\,d-1}\,
L\,\ln\frac{\alpha c}{L},
\]
which is the corollary.
\end{proof}

\begin{proof}[Proof of Corollary~\ref{cor:matched}]
Take $\alpha=4$ in Corollary~\ref{cor:kdl}, whose hypotheses then read
$c\ge2$ and $4^{-c}<\eps<\frac12$:
\[
\Lambda_\eps(cA_0,B_0)\;\le\;C\bigr|_{\alpha=4}\;c^{\,d-1}\,L\,
\ln\frac{4c}{L}.
\]
It remains to replace $L\ln(4c/L)$ by $\Ltil\log_2(4c/\Ltil)$. Write
$Y=\log_2\frac{4c}{\Ltil}$, so that $\ln\frac{4c}{\Ltil}=Y\ln2$.

First, $\Ltil\le2L$ by \eqref{eq:Lchain}, so
$\frac{4c}{L}=\frac{4c}{\Ltil}\cdot\frac{\Ltil}{L}\le\frac{8c}{\Ltil}$ and
therefore
\[
\ln\frac{4c}{L}\;\le\;\ln\frac{4c}{\Ltil}+\ln2 .
\]
Second, the hypothesis $\eps>4^{-c}$ gives $L<c\ln4$, whence
$\frac{4c}{\Ltil}\ge\frac{4c}{2L}>\frac{4c}{2c\ln4}=\frac{2}{\ln4}$ and
\[
\ln\frac{4c}{\Ltil}\;>\;\ln\frac{2}{\ln4}=0.36651\ldots\;>\;0.3665 .
\]
The first display, rewritten multiplicatively and then estimated by the
second, gives
\[
\ln\frac{4c}{L}\;\le\;\Bigl(1+\frac{\ln2}{\ln(4c/\Ltil)}\Bigr)
\ln\frac{4c}{\Ltil}
\;\le\;\Bigl(1+\frac{\ln2}{0.3665}\Bigr)\ln\frac{4c}{\Ltil}
\;\le\;2.892\,\ln\frac{4c}{\Ltil}
\;=\;2.892\ln2\cdot Y\;\le\;2.005\,Y .
\]
Finally $L\le\Ltil$, again by \eqref{eq:Lchain}. Hence
$L\ln\frac{4c}{L}\le2.005\,\Ltil\,Y\le2.01\,\Ltil\,Y$, and the corollary
follows with $C'=2.01\,C\bigr|_{\alpha=4}$.
\end{proof}

\begin{remark}[Shape of the bound]
All logarithms above are natural; restating \eqref{eq:conj} in any other
fixed base changes only the constant. No factor in Theorem~\ref{thm:upper}
can be removed: for $A_0=B_0=[0,1]^d$, \cite[Thm.~1.1]{KDL} exhibits a
matching lower bound $\Lambda^\pm_\eps\gtrsim c^{d-1}LR$, proved there on the
deep polynomial range $\alpha_d^{-c}<\eps<c^{-\alpha_d}$. The lower-bound
half of that theorem carries the standing hypothesis $\eps>\alpha_d^{-c}$ of
the first half, since below $\alpha_d^{-c}$ the conclusion is replaced by
\cite[(1.3)]{KDL}. So the shape $c^{d-1}LR$ is attained in that range, and
Theorem~\ref{thm:upper} cannot be improved in shape there.
\end{remark}

\section{Unconditional lower bounds}\label{sec:lower}

Throughout this section the geometry is the model
box pair $A_0=B_0=(0,1)^d$ with $c=\ell$, for which
(Remark~\ref{rem:cube}) the concentration operator is the exact tensor
power $S_\ell^{\otimes d}$ of the one-dimensional operator \eqref{eq:Sell}.
Its eigenvalues are the $d$-fold products $\prod_{i=1}^d\lambda_{n_i}(\ell)$,
so that
\begin{equation}\label{eq:prodcount}
\Lambda_\eps(c)=\#\Bigl\{(n_1,\dots,n_d):\
\eps<\prod_{i=1}^d\lambda_{n_i}(\ell)<1-\eps\Bigr\},
\end{equation}
and every bound below is obtained by exhibiting a set of multi-indices whose
product lies in $(\eps,1-\eps)$. The task is therefore one-dimensional
throughout: to show that the sinc-kernel spectrum really does place many
eigenvalues in a fixed window. Sections~\ref{sec:tracestep}--\ref{sec:tier1}
do this from the exact trace identity together with the sine-kernel
determinant asymptotics of Appendix~\ref{app:exact};
Sections~\ref{sec:fixeddepth}--\ref{sec:edge} draw the consequences.

Every statement of this section is proved in one and the same sense: proved
here modulo the single quoted external theorem, the sine-kernel determinant
asymptotics of Theorem~\ref{thm:basorwidom}.
The two traces the window count consumes, at $m=2$ and $m=3$, are settled
with their constant terms in closed form in Appendix~\ref{app:exact}; the
higher traces at every other fixed order are settled by
Proposition~\ref{lem:higher} below, from the same machinery. What is not
proved, and is claimed nowhere, is any estimate uniform in the order $m$ or
in the depth $u$ of Section~\ref{sec:fixeddepth};
Remark~\ref{rem:mquant} locates that limitation.

\subsection{The trace identity and its two-sided bound}\label{sec:tracestep}

Write $\gamma_{\mathrm E}$ for Euler's constant and $\operatorname{Ci}$ for
the cosine integral.

\begin{proposition}[Escape identity and the $\Tr(S-S^2)$
asymptotics]\label{prop:trace}
For every $\ell>0$,
\begin{equation}\label{eq:escape}
\Tr(S_\ell-S_\ell^2)=\iint_{I\times I^c}\sinc(x-y)^2\,dx\,dy
=2\int_0^\ell g(x)\,dx,\qquad
g(x):=\int_x^\infty\sinc(t)^2\,dt .
\end{equation}
Consequently
\begin{equation}\label{eq:trphi}
\Tr(S_\ell-S_\ell^2)=\frac1{\pi^2}
\bigl(\ln\ell+\gamma_{\mathrm E}+\ln2\pi-\operatorname{Ci}(2\pi\ell)\bigr)
+2\ell\,g(\ell),
\end{equation}
and there are explicit constants with, for all $\ell\ge1$,
\begin{equation}\label{eq:trbounds}
\frac1{\pi^2}\ln\ell+0.22\ \le\ \Tr(S_\ell-S_\ell^2)\ \le\
\frac1{\pi^2}\ln\ell+0.47 .
\end{equation}
In particular $\Tr(S_\ell-S_\ell^2)=\frac1{\pi^2}\ln\ell+O(1)$, with leading
constant exactly $1/\pi^2$.
\end{proposition}

\begin{proof}
$\Tr S_\ell=\int_Iq(x,x)\,dx=\ell$ and
$\Tr S_\ell^2=\iint_{I\times I}q(x,y)^2\,dx\,dy$. By the second identity in
\eqref{eq:proj}, $\ell=\int_I\bigl(\int_\R q(x,y)^2dy\bigr)dx$, so
$\Tr(S_\ell-S_\ell^2)=\int_I\int_{I^c}q(x,y)^2\,dy\,dx$, which is the first
equality in \eqref{eq:escape}. Splitting
$I^c=(-\infty,0)\cup(\ell,\infty)$ and substituting $t=x-y$, resp.\
$t=y-x$, gives $\int_{I^c}q(x,y)^2dy=g(x)+g(\ell-x)$, and
$\int_0^\ell[g(x)+g(\ell-x)]dx=2\int_0^\ell g$. Integrating by parts with
$g'(x)=-\sinc(x)^2$,
\[
\int_0^\ell g=\ell\,g(\ell)+\int_0^\ell x\,\sinc(x)^2dx
=\ell g(\ell)+\frac1{\pi^2}\int_0^\ell\frac{\sin^2\pi x}{x}\,dx
=\ell g(\ell)+\frac1{2\pi^2}
\bigl(\gamma_{\mathrm E}+\ln2\pi\ell-\operatorname{Ci}(2\pi\ell)\bigr),
\]
using
$\int_0^\ell\frac{1-\cos2\pi x}{2x}dx
=\frac12(\gamma_{\mathrm E}+\ln2\pi\ell-\operatorname{Ci}(2\pi\ell))$;
doubling gives \eqref{eq:trphi}. For \eqref{eq:trbounds}:
$g(\ell)=\int_\ell^\infty\sinc^2\ge0$ and
$g(\ell)\le\int_\ell^\infty\frac{dt}{\pi^2t^2}=\frac1{\pi^2\ell}$, so
$0\le2\ell g(\ell)\le\frac2{\pi^2}$; also
$|\operatorname{Ci}(2\pi\ell)|\le\frac1{2\pi\ell}\le\frac1{2\pi}$ for
$\ell\ge1$, and $\gamma_{\mathrm E}+\ln2\pi=2.4151\ldots$ Hence the constant
term lies between
$\frac1{\pi^2}(2.4151-\frac1{2\pi})=0.22857\ldots$ and
$\frac1{\pi^2}(2.4151+\frac1{2\pi})+\frac2{\pi^2}=0.46347\ldots$; rounding
\emph{outward} to $[0.22,\,0.47]$ gives \eqref{eq:trbounds}.
\end{proof}

The interval in \eqref{eq:trbounds} is the one the proof establishes, and it
is two-sided because both endpoints are rounded outward. Its interior
contains the limiting value of the constant term. Indeed, in
\eqref{eq:trphi} $\operatorname{Ci}(2\pi\ell)\to0$ and
$2\ell g(\ell)\to\pi^{-2}$ as $\ell\to\infty$, so the constant term tends to
$(1+\gamma_{\mathrm E}+\ln2\pi)/\pi^2=0.3460\ldots$, and a limit is no bound
at finite $\ell$. The interval is not narrowed on its account.

\subsection{The moment count inequality}\label{sec:moment}

Since $\lambda\in(\frac14,\frac34)\iff\varphi=\lambda(1-\lambda)>\frac3{16}$,
the target count \eqref{eq:Ma} is $M_a=\#\{n:\varphi_n>\frac3{16}\}$, where
$\varphi_n=\lambda_n(1-\lambda_n)$ as in Section~\ref{sec:1d}. It is reached
by a polynomial that is dominated by the indicator of the $\varphi$-window
and whose traces are computable.

\begin{proposition}[Exact count inequality]\label{prop:count}
Let $\varrho(w)=16w^2(16w-3)=256w^3-48w^2$. Then
$\varrho(w)\le\ind{(3/16,\,1/4]}(w)$ for all $w\in[0,\frac14]$, and therefore
\begin{equation}\label{eq:countineq}
M_a\;=\;\#\{n:\lambda_n(\ell)\in(\tfrac14,\tfrac34)\}\ \ge\
256\,\Tr\bigl((S_\ell-S_\ell^2)^3\bigr)
-48\,\Tr\bigl((S_\ell-S_\ell^2)^2\bigr).
\end{equation}
\end{proposition}

\begin{proof}
$\varrho(0)=0$, and $\varrho(w)=16w^2(16w-3)$ has the sign of $16w-3$ on
$(0,\frac14]$: thus $\varrho\le0$ on $[0,\frac3{16}]$ and $\varrho>0$ on
$(\frac3{16},\frac14]$. On the latter interval $\varrho'(w)=96w(8w-1)>0$
(as $w>\frac18$), so $\varrho$ increases from $0$ to
$\varrho(\frac14)=256\cdot\frac1{64}-48\cdot\frac1{16}=4-3=1$. Hence
$0\le\varrho\le1$ there and $\varrho\le\ind{(3/16,1/4]}$ everywhere on
$[0,\frac14]$. Evaluating at $w=\varphi_n\in[0,\frac14]$ and summing,
$\sum_n\varrho(\varphi_n)\le\#\{n:\varphi_n>\frac3{16}\}=M_a$; the left side
is $256\Tr(\varphi^3)-48\Tr(\varphi^2)$ because
$\Tr(\varphi^m)=\sum_n\varphi_n^m$.
\end{proof}

Inequality \eqref{eq:countineq} is exact. Of its two polynomial constraints,
$\varrho\le0$ on $[0,\frac3{16}]$ and $\varrho\le1$ on
$(\frac3{16},\frac14]$, only the second is attained, and only at the right
endpoint, where $\varrho(\frac14)=1$.

\subsection{The higher traces at every fixed order}\label{sec:higher}

The window count of Section~\ref{sec:tier1} consumes $\Tr(\varphi^m)$ at
$m=2$ and $m=3$ only, and Appendix~\ref{app:exact} settles those two with
their constant terms in closed form. The same expansion settles \emph{every}
fixed $m$, and it is included here because Section~\ref{sec:fixeddepth}
consumes it at larger $m$. The route is short: Lemma~\ref{lem:extract} of
Appendix~\ref{app:exact} gives $\Tr(S_\ell^{\,k})$ at every fixed $k$, and
$\Tr(\varphi^m)$ is a fixed finite combination of those.

\begin{proposition}[fixed-$m$ higher-trace asymptotics]\label{lem:higher}
Assume Theorem~\ref{thm:basorwidom}. For each fixed integer $m\ge1$,
\begin{equation}\label{eq:higherm}
\Tr\bigl((S_\ell-S_\ell^2)^m\bigr)
=\frac{\beta_m}{\pi^2}\ln\ell+O_m(1)\qquad(\ell\ge1),
\end{equation}
where
\[
\beta_m=B(m,m)=\frac{((m-1)!)^2}{(2m-1)!},
\qquad\text{so}\qquad
\beta_1=1,\quad\beta_2=\tfrac16,\quad\beta_3=\tfrac1{30},
\]
the $O_m(1)$ term being bounded uniformly in $\ell\ge1$ \emph{for each fixed
$m$}. No bound uniform in $m$ is asserted and none is proved; see
Remark~\ref{rem:mquant}.
\end{proposition}

\begin{proof}
Fix $m\ge1$; every step is at that fixed $m$.

\emph{Step 1 (the exact finite combination).} $S_\ell$ commutes with
$I-S_\ell$, so $\varphi^m=\bigl(S_\ell(I-S_\ell)\bigr)^m
=S_\ell^{\,m}(I-S_\ell)^m$, and the binomial theorem in the commutative
algebra generated by $S_\ell$ gives
\begin{equation}\label{eq:phicomb}
\varphi^m=\sum_{j=0}^m(-1)^j\binom mj S_\ell^{\,m+j},
\qquad\text{hence}\qquad
\Tr(\varphi^m)=\sum_{j=0}^m(-1)^j\binom mj\Tr\bigl(S_\ell^{\,m+j}\bigr).
\end{equation}
Every term is finite: $S_\ell$ is trace class with $\Tr S_\ell=\ell$, and
$0\le S_\ell\le I$ gives $S_\ell^{\,k}\le S_\ell$, so
$\Tr(S_\ell^{\,k})\le\ell$ for every $k\ge1$. The sum has $m+1$ terms and its
highest power is $2m$; both are fixed, which is the whole point.

\emph{Step 2 (the range $\ell\ge\ell_1$).} Lemma~\ref{lem:extract} applies at
each of the $m+1$ fixed exponents $k=m,m+1,\dots,2m$ and gives, for
$\ell\ge\ell_1$,
\[
\Tr(S_\ell^{\,k})=\ell-\frac{\mathsf H_{k-1}}{\pi^2}\ln\ell+\mathsf c_k
+\mathsf R_k(\ell),
\qquad
|\mathsf R_k(\ell)|\le k\bigl(\tfrac{16}3\bigr)^k
\delta_\ell^{\,\vartheta}\mathsf M_\ell^{\,1-\vartheta},
\]
where $\ell_1\ge1$ is a threshold beyond which hypotheses~(i) and~(ii) of
that lemma both hold: (ii) holds for every $\ell\ge1$ by
Lemma~\ref{lem:crude}, and (i) is Theorem~\ref{thm:basorwidom}, with
$\delta_\ell=O(\ln\ell/\ell)$. Substituting into \eqref{eq:phicomb} splits
$\Tr(\varphi^m)$ into four sums.

\emph{The terms linear in $\ell$ cancel:} their coefficient is
$\sum_{j=0}^m(-1)^j\binom mj=(1-1)^m=0$, because $m\ge1$.

\emph{The logarithmic coefficient is exactly $\beta_m/\pi^2$:} it equals
$-\frac1{\pi^2}\sum_{j=0}^m(-1)^j\binom mj\mathsf H_{m+j-1}$, and
Lemma~\ref{lem:harmonic} of Appendix~\ref{app:exact} evaluates that
alternating sum as $-B(m,m)=-\beta_m$, giving $+\beta_m/\pi^2$.

\emph{The constant term is a fixed finite number,}
$\mathsf K_m:=\sum_{j=0}^m(-1)^j\binom mj\mathsf c_{m+j}$, each
$\mathsf c_k=-k\,[\sigma^k]\mathsf C(\sigma)$ being a Taylor coefficient of
the function $\mathsf C$, which is analytic on $|\sigma|<1$, as the proof
of Lemma~\ref{lem:extract} records, from the non-vanishing established in
the proof of Lemma~\ref{lem:crude}. So $\mathsf c_k$ is well defined at every
$k\ge1$, and $\mathsf K_m$ is a finite sum of $m+1$ of them.

\emph{The remainder tends to zero:} it is
$\sum_{j=0}^m(-1)^j\binom mj\mathsf R_{m+j}(\ell)$, of modulus at most
\[
\sum_{j=0}^m\binom mj(m+j)\bigl(\tfrac{16}3\bigr)^{m+j}
\delta_\ell^{\,\vartheta}\mathsf M_\ell^{\,1-\vartheta}
\ \le\ 2m\cdot2^m\bigl(\tfrac{16}3\bigr)^{2m}
\delta_\ell^{\,\vartheta}\mathsf M_\ell^{\,1-\vartheta}
\ =\ O_m\bigl(\ell^{-\vartheta+o(1)}\bigr),
\]
by \eqref{eq:starsub}. The prefactor is finite at each fixed $m$; it is the
only place $m$ enters, and it grows like $(16/3)^{2m}$.

Hence \eqref{eq:higherm} holds on $[\ell_1,\infty)$, with an $O_m(1)$ that in
fact converges to $\mathsf K_m$.

\emph{Step 3 (the range $1\le\ell\le\ell_1$).} No asymptotics are needed
here, and the extension is by a crude two-sided bound, not by continuity,
compactness or monotonicity. The eigenvalues of $\varphi$ lie in
$[0,\frac14]$, so $\varphi_n^{\,m}\le(\frac14)^{m-1}\varphi_n$ termwise and
\[
0\ \le\ \Tr(\varphi^m)\ \le\ \bigl(\tfrac14\bigr)^{m-1}\Tr(\varphi)
\ \le\ \bigl(\tfrac14\bigr)^{m-1}
\Bigl(\frac{\ln\ell}{\pi^2}+0.47\Bigr),
\]
the last step by \eqref{eq:trbounds}, which is valid for every $\ell\ge1$.
Since $0\le\ln\ell\le\ln\ell_1$ on $[1,\ell_1]$, the triangle inequality
gives there
\[
\Bigl|\Tr(\varphi^m)-\frac{\beta_m}{\pi^2}\ln\ell\Bigr|
\ \le\ \Bigl(\bigl(\tfrac14\bigr)^{m-1}+\beta_m\Bigr)
\frac{\ln\ell_1}{\pi^2}+\bigl(\tfrac14\bigr)^{m-1}\cdot0.47,
\]
a finite bound depending on $m$ and $\ell_1$ alone. Taking the larger of the
two bounds gives \eqref{eq:higherm} for every $\ell\ge1$.

\emph{Step 4 (consistency with the two evaluated cases).} At $m=1$,
$\beta_1=1$ and \eqref{eq:higherm} is Proposition~\ref{prop:trace}, proved
above independently of this appendix and with the explicit two-sided
constants \eqref{eq:trbounds}. At $m=2,3$, Lemma~\ref{lem:exacttrace}
evaluates the constants $\mathsf K_2,\mathsf K_3$ in closed form and sharpens
$O_m(1)$ to a power-saving remainder; its leading coefficients
$\frac1{6\pi^2}$ and $\frac1{30\pi^2}$ are $\beta_2/\pi^2$ and
$\beta_3/\pi^2$.
\end{proof}

\noindent
The Beta value in $\beta_m$ is not an accident of the harmonic-number
algebra, and it has a second reading. For $m\ge2$ the trace is
the convergent multiple sinc-integral
$\Tr(\varphi^m)=\int_{I^m}\prod_i\varphi(x_i,x_{i+1})$ with
$\varphi(x,z)=\int_{I^c}q(x,y)q(y,z)\,dy$; the $\ln\ell$ arises from the bulk
overlap exactly as for $m=1$, and the pushforward of the spectral measure to
$w=\varphi$ carries plunge weight $\frac{d\lambda}{\lambda(1-\lambda)}$,
giving
\[
\frac1{\pi^2}\int_0^1[\lambda(1-\lambda)]^{m-1}\,d\lambda
=\frac1{\pi^2}B(m,m)=\frac{\beta_m}{\pi^2}.
\]
That reading is a heuristic and is not what proves
Proposition~\ref{lem:higher}; Lemma~\ref{lem:harmonic} is. The two arrive at
the same Beta integral by routes sharing no step: one from the spectral side,
the other from the alternating sum of harmonic numbers.

\subsection{The one-dimensional window count}\label{sec:tier1}

\begin{theorem}[Window count]\label{thm:window}
There is an explicit constant $c_0=\frac{8}{15\pi^2}$, and a
finite constant $C_0$, with
\begin{equation}\label{eq:tier1}
M_a=\#\{n:\lambda_n(\ell)\in(\tfrac14,\tfrac34)\}\ \ge\ c_0\ln\ell-C_0
\qquad(\ell>0).
\end{equation}
\end{theorem}

\begin{proof}
The count inequality \eqref{eq:countineq} is unconditional, and it reduces
the theorem to the asymptotics of $\Tr(\varphi^2)$ and $\Tr(\varphi^3)$, the
only two traces the degree-three moment polynomial $256w^3-48w^2$ touches.
Both are supplied by Lemma~\ref{lem:exacttrace} of Appendix~\ref{app:exact},
which at $m=2,3$ gives more than Proposition~\ref{lem:higher} does: the
constant terms in closed form and a power-saving remainder,
\[
\Tr(\varphi^2)=\frac{\beta_2}{\pi^2}\ln\ell+\mathsf K_2+o(1),
\qquad
\Tr(\varphi^3)=\frac{\beta_3}{\pi^2}\ln\ell+\mathsf K_3+o(1),
\]
with $\beta_2=\frac16$, $\beta_3=\frac1{30}$ and with $\mathsf K_2$,
$\mathsf K_3$ in closed form, conditional on the cited determinant
asymptotics (Theorem~\ref{thm:basorwidom}) and on nothing else. Combining
with \eqref{eq:countineq},
\[
M_a\ \ge\ 256\Tr(\varphi^3)-48\Tr(\varphi^2)
=\frac{256\beta_3-48\beta_2}{\pi^2}\ln\ell+\mathsf B+o(1),
\]
where
$256\beta_3-48\beta_2=\frac{256}{30}-\frac{48}{6}=\frac{128}{15}-8=\frac8{15}$,
whence the leading constant is $c_0=\frac8{15\pi^2}$; this is
Corollary~\ref{cor:Bconst}. Corollary~\ref{cor:windowexact} converts that
asymptotic into \eqref{eq:tier1} for every $\ell>0$, with the same $c_0$ and
a finite $C_0$. (Because the left side of \eqref{eq:tier1}
is a nonnegative integer, the inequality holds automatically wherever
$c_0\ln\ell\le C_0$, so only a finite remainder constant is needed; that is
how Corollary~\ref{cor:windowexact} covers the compact range $\ell<\ell_0$
as well.)
\end{proof}

\noindent
The two constants have different standing, and the statement is worded to
keep them apart. The coefficient $c_0$ is exact and is given in closed form;
Remark~\ref{rem:c0status} says in what sense it is optimal and in what sense
it is not. The remainder constant $C_0$ is finite but \emph{not effective}.
By Corollary~\ref{cor:windowexact} it is $c_0\ln\ell_0$ for a threshold
$\ell_0$ beyond which the remainder of Lemma~\ref{lem:exacttrace} is at most
half of its constant term, and that threshold rests on an implied constant
which the cited source states to be uniform but does not quantify.

\begin{remark}[in what sense $c_0$ is, and is not, optimal]\label{rem:c0status}
Three constants must be kept apart, and only the first two are claimed here.

\emph{(1) The coefficient of the chosen polynomial.} Fixing
$\varrho(w)=256w^3-48w^2$ of Proposition~\ref{prop:count} forces
$\mathcal I[\varrho]=\pi^{-2}(256\beta_3-48\beta_2)=\frac8{15\pi^2}=c_0$
exactly, by \eqref{eq:Ifunctional}. Nothing is lost in passing from the
polynomial to the coefficient: $c_0$ is \emph{the exact coefficient delivered
by that degree-three minorant}, and Corollary~\ref{cor:windowexact} carries
it through to \eqref{eq:tier1} without shaving anything off it.

\emph{(2) The maximum over minorants of degree at most three.} This is again
$c_0$, and it is \emph{proved}, in Lemma~\ref{lem:degree3}: a minorant of
degree at most two yields at most $0$, and among minorants of degree at most
three admissible for $(\frac3{16},\frac14]$ the value $\frac8{15\pi^2}$ is
attained and not improvable. The proof is a three-term dual certificate,
\eqref{eq:dualcert3}, checkable by hand. So $c_0$ may fairly be called
optimal \emph{at degree three}, and that is the only optimality asserted for
it anywhere in this paper.

\emph{(3) The all-degree ceiling, which is also the fixed-window asymptotic
coefficient.} This is \emph{not} $c_0$. By Lemma~\ref{lem:minorant}, applied
with $b=\frac14$ as in Remark~\ref{rem:minorantendpoints}, the supremum of
$\mathcal I[\varrho]$ over \emph{all} admissible minorants, of any degree, is
\[
\mathcal I^\ast\bigl(\tfrac3{16},\tfrac14\bigr)
=\frac1{\pi^2}\int_{1/4}^{3/4}\frac{d\lambda}{\lambda(1-\lambda)}
=\frac{2\ln3}{\pi^2}=0.22262\ldots,
\]
larger than $c_0$ by the exact factor $\frac{15}4\ln3$. The same number
is the coefficient in the Landau--Widom asymptotics at this fixed window,
$M_a=\frac{2\ln3}{\pi^2}\ln\ell+o(\ln\ell)$ \cite{LandauWidom}, and the two
normalizations agree for a reason internal to this paper: the Landau--Widom
plunge density in the present scaling is
$\pi^{-2}\ln\ell\cdot d\lambda/(\lambda(1-\lambda))$, whose moment against
$[\lambda(1-\lambda)]^m$ is
$\pi^{-2}\ln\ell\int_0^1[\lambda(1-\lambda)]^{m-1}d\lambda
=\beta_m\pi^{-2}\ln\ell$, the leading term proved in
Proposition~\ref{lem:higher}, the case $m=1$ being
Proposition~\ref{prop:trace}. So the polynomial-minorant method is not lossy
in the limit: its ceiling is the fixed-window asymptotic coefficient, while
degree three realizes about a quarter of it, and by
Lemma~\ref{lem:minorant}(ii) every fraction below $1$ is realized by some
admissible polynomial of sufficiently large degree. Two things are
deliberately not said. No specific higher-degree minorant is exhibited
together with a feasibility certificate, so no coefficient at any degree
above three is asserted here. And $2\ln3/\pi^2$ is not called an optimal
nonasymptotic constant for $M_a$; it is the supremum of the leading
coefficients over the class of Lemma~\ref{lem:minorant}, it is the
fixed-window Landau--Widom asymptotic coefficient, and this paper uses it as
nothing else. Degree three is retained in Theorem~\ref{thm:window} because it
is the one minorant that can be displayed outright and evaluated against
traces whose constant terms are known in closed form
(Lemma~\ref{lem:exacttrace}).
\end{remark}

\begin{lemma}[the moment method at degrees two and three]\label{lem:degree3}
Call a real polynomial $\varrho$ \emph{admissible} if $\varrho(0)=0$ and
$\varrho(w)\le\ind{(3/16,\,1/4]}(w)$ for every $w\in[0,\frac14]$; this is the
constraint established in Proposition~\ref{prop:count}, and it is
Definition~\ref{def:admissible} below specialized to the window
$(\frac3{16},\frac14]$. For $\varrho(w)=\sum_{m\ge1}b_mw^m$ write
$\mathcal I[\varrho]=\pi^{-2}\sum_{m\ge1}b_m\beta_m$ with
$\beta_m=B(m,m)$, so that $\mathcal I[\varrho]$ is the coefficient of
$\ln\ell$ that $\varrho$ delivers through \eqref{eq:countineq} and
Proposition~\ref{lem:higher}. Then:
\begin{enumerate}[label=(\roman*),itemsep=3pt,topsep=3pt]
\item if $\deg\varrho\le2$ then $\mathcal I[\varrho]\le0$, with equality for
$\varrho\equiv0$;
\item if $\deg\varrho\le3$ then
$\mathcal I[\varrho]\le c_0=\frac8{15\pi^2}$, with equality for
$\varrho(w)=256w^3-48w^2$.
\end{enumerate}
\end{lemma}

\begin{proof}
Admissibility enters at exactly three points of $[0,\frac14]$, and only
through the three consequences
\begin{equation}\label{eq:threeconstraints}
b_1\le0,\qquad \varrho\bigl(\tfrac3{16}\bigr)\le0,\qquad
\varrho\bigl(\tfrac14\bigr)\le1 .
\end{equation}
The second and third are the defining inequality evaluated at $w=\frac3{16}$,
where the indicator vanishes because the window is open on the left, and at
$w=\frac14$, where it equals $1$. For the first, $\varrho(w)\le0$ and $w>0$
on $(0,\frac3{16}]$ give $\varrho(w)/w=b_1+b_2w+b_3w^2\le0$ there; letting
$w\downarrow0$ yields $b_1\le0$.

Recall $\beta_1=1$, $\beta_2=\frac16$, $\beta_3=\frac1{30}$
(Proposition~\ref{prop:trace} and Lemma~\ref{lem:exacttrace}).

(i) With $b_3=0$ one has the identity, valid for all $(b_1,b_2)$,
\begin{equation}\label{eq:dualcert2}
b_1+\tfrac16b_2\;=\;\tfrac19\,b_1
\;+\;\tfrac{128}{27}\,\varrho\bigl(\tfrac3{16}\bigr),
\end{equation}
since the coefficient of $b_1$ on the right is
$\frac19+\frac{128}{27}\cdot\frac3{16}=\frac19+\frac89=1$ and that of $b_2$
is $\frac{128}{27}\cdot\frac9{256}=\frac16$. The multipliers $\frac19$ and
$\frac{128}{27}$ are positive, so \eqref{eq:threeconstraints} forces
$\pi^2\mathcal I[\varrho]=b_1+\frac16b_2\le0$. The zero polynomial is
admissible and attains $0$.

(ii) The corresponding identity at degree three, valid for all
$(b_1,b_2,b_3)$, is
\begin{equation}\label{eq:dualcert3}
b_1+\tfrac16b_2+\tfrac1{30}b_3
\;=\;\tfrac7{45}\,b_1
\;+\;\tfrac{512}{135}\,\varrho\bigl(\tfrac3{16}\bigr)
\;+\;\tfrac8{15}\,\varrho\bigl(\tfrac14\bigr),
\end{equation}
as one checks by collecting the three coefficients separately:
\[
\tfrac7{45}+\tfrac{512}{135}\cdot\tfrac3{16}+\tfrac8{15}\cdot\tfrac14
=\tfrac7{45}+\tfrac{32}{45}+\tfrac6{45}=1,
\]
\[
\tfrac{512}{135}\cdot\tfrac9{256}+\tfrac8{15}\cdot\tfrac1{16}
=\tfrac2{15}+\tfrac1{30}=\tfrac16,
\qquad
\tfrac{512}{135}\cdot\tfrac{27}{4096}+\tfrac8{15}\cdot\tfrac1{64}
=\tfrac1{40}+\tfrac1{120}=\tfrac1{30}.
\]
The three multipliers $\frac7{45}$, $\frac{512}{135}$, $\frac8{15}$ are
positive, so \eqref{eq:threeconstraints} and \eqref{eq:dualcert3} give
\[
\pi^2\,\mathcal I[\varrho]\ \le\ \tfrac7{45}\cdot0
+\tfrac{512}{135}\cdot0+\tfrac8{15}\cdot1=\tfrac8{15},
\]
that is, $\mathcal I[\varrho]\le\frac8{15\pi^2}=c_0$. Equality holds for
$\varrho(w)=256w^3-48w^2$, which is admissible by
Proposition~\ref{prop:count} and meets all three constraints of
\eqref{eq:threeconstraints} with equality: $b_1=0$,
$\varrho(\frac3{16})=0$, $\varrho(\frac14)=1$.
\end{proof}

\noindent
Identity \eqref{eq:dualcert3} is a dual feasible point for the linear
program \emph{maximize $\mathcal I[\varrho]$ over admissible $\varrho$ of
degree at most three}: the nonnegative triple
$(\frac7{45},\frac{512}{135},\frac8{15})$ reproduces the moment vector
$(\beta_1,\beta_2,\beta_3)$ from evaluation at $0^+$, at $\frac3{16}$ and at
$\frac14$, and its value against the indicator is $\frac8{15}$. Optimality
therefore requires no search over polynomials: three rational multipliers and
three arithmetic identities settle it.
Part~(i) is the reason the third moment is needed at all. No combination of
the global traces $\Tr S_\ell=\ell$ and $\Tr S_\ell^2$ separates the band
mass from the plunge tail, whose mass is $\Theta(\ln\ell)$ and comparable to
it; the third moment is the first that resolves them, which is why the naive
$\Tr(S-S^2)$/$\Tr((S-S^2)^2)$ Paley--Zygmund route does not close and
\eqref{eq:countineq} does.

\subsection{Window density at fixed depth}\label{sec:fixeddepth}

Write the depth of an eigenvalue as $u(\lambda)=\ln\frac{1-\lambda}\lambda$,
so that $u>0$ below $\frac12$ and $u<0$ above it, and, for $u>0$, set
\[
W_0(u,\ell)=\#\{n:\lambda_n\in(e^{-2u},e^{-u})\},\qquad
W_1(u,\ell)=\#\{n:\lambda_n\in(1-e^{-u},1-e^{-2u})\}.
\]
Both are indexed by the \emph{magnitude} of the depth: $W_0$ collects the
eigenvalues at depths in $(u,2u)$, $W_1$ their reflections, at depths in
$(-2u,-u)$. The two ranges are disjoint, and hence so are the two
$\lambda$-windows above, precisely when $u>\ln2$; that is why the proposition
is stated on that range. At $u=\ln2$ the windows abut at $\lambda=\frac12$,
reading $(\frac14,\frac12)$ and $(\frac12,\frac34)$, and for $u<\ln2$ they
overlap and $W_0+W_1$ double-counts.

The moment method needs a polynomial that lies below the indicator of the
$\varphi$-window and vanishes at the origin. For the window $(\frac3{16},
\frac14]$ of Section~\ref{sec:tier1} one such polynomial was displayed
outright (Proposition~\ref{prop:count}). For a general interior window no
single polynomial presents itself, and the existence of good ones has to be
proved. That is done first: without it Proposition~\ref{prop:fixeddepth}
would not be unconditional. The construction is standard approximation
theory, and it is carried out here rather than asserted.

\begin{definition}[admissible minorant]\label{def:admissible}
Let $0<a<b\le\frac14$. A real polynomial $\varrho$ is \emph{admissible for
$(a,b)$} if
\begin{enumerate}[label=(A\arabic*),itemsep=1pt,topsep=3pt]
\item $\varrho(0)=0$;
\item $\varrho(w)\le\ind{(a,b)}(w)$ for every $w\in[0,\frac14]$.
\end{enumerate}
For such a $\varrho$, written $\varrho(w)=\sum_{m=1}^{m_0}b_mw^m$, put
\begin{equation}\label{eq:Ifunctional}
\mathcal I[\varrho]\ :=\ \frac1{\pi^2}\sum_{m=1}^{m_0}b_m\beta_m
\ =\ \frac1{\pi^2}\int_0^1\varrho\bigl(\lambda(1-\lambda)\bigr)
\frac{d\lambda}{\lambda(1-\lambda)},
\end{equation}
the second equality because
$\int_0^1[\lambda(1-\lambda)]^{m-1}d\lambda=B(m,m)=\beta_m$.
\end{definition}

\noindent
(A1) is not a normalization: it is what makes
$\sum_n\varrho(\varphi_n)=\sum_mb_m\Tr(\varphi^m)$ a \emph{finite} sum of
traces. A constant term $b_0\ne0$ would contribute $b_0\sum_n1=\pm\infty$,
since $\varphi$ has infinite-dimensional kernel; with $b_0=0$ every term is
$\Tr(\varphi^m)\le\frac14^{\,m-1}\Tr\varphi<\infty$ for $m\ge1$, and the sum
converges absolutely. (A2) is one-sided only: $\varrho$ may be negative
anywhere, and the admissible polynomials constructed below are. What (A2)
buys is the inequality
\begin{equation}\label{eq:minorantcount}
\sum_{m}b_m\Tr(\varphi^m)=\sum_n\varrho(\varphi_n)
\ \le\ \#\{n:\varphi_n\in(a,b)\},
\end{equation}
valid termwise because $\varphi_n\in[0,\frac14]$ for every $n$.

\begin{lemma}[polynomial minorants of an interior spectral
window]\label{lem:minorant}
Let $0<a<b<\frac14$. Then:
\begin{enumerate}[label=(\roman*),itemsep=3pt,topsep=3pt]
\item \emph{(ceiling)} $\mathcal I[\varrho]\le\mathcal I^\ast(a,b)$ for every
$\varrho$ admissible for $(a,b)$, where
\[
\mathcal I^\ast(a,b):=\frac1{\pi^2}\int_{\{\lambda\in(0,1):\,
a<\lambda(1-\lambda)<b\}}\frac{d\lambda}{\lambda(1-\lambda)}\ <\ \infty .
\]
\item \emph{(attainment in the limit)} There is a sequence $(\varrho_N)$ of
polynomials admissible for $(a,b)$, of finite degrees $m_0(N)$, with
$\mathcal I[\varrho_N]\to\mathcal I^\ast(a,b)$.
\end{enumerate}
Consequently $\sup\{\mathcal I[\varrho]:\varrho\text{ admissible for
}(a,b)\}=\mathcal I^\ast(a,b)$, and for every $\theta<1$ there is an
admissible polynomial of finite degree with
$\mathcal I[\varrho]\ge\theta\,\mathcal I^\ast(a,b)$.
\end{lemma}

\begin{proof}
Write $w=\lambda(1-\lambda)$ and $E=\{\lambda\in(0,1):a<w<b\}$.
$\mathcal I^\ast$ is finite because $w>a>0$ on $E$ and $|E|\le1$, so the
integrand is at most $1/a$.

(i) By (A2), $\varrho(w)\le\ind{(a,b)}(w)$ pointwise on $[0,\frac14]$, and
$\lambda\mapsto w$ maps $(0,1)$ into $(0,\frac14]$; the measure
$d\lambda/w$ is positive on $(0,1)$. The integral in \eqref{eq:Ifunctional}
converges absolutely, since $\varrho(w)/w$ is a polynomial in $w$ and hence
bounded on $[0,\frac14]$, which is (A1) again. Integrating the pointwise
inequality gives $\mathcal I[\varrho]\le\mathcal I^\ast(a,b)$.

(ii) In four steps.

\emph{Step 1 (continuous minorants of the indicator).} For each integer
$N\ge1$ set
\[
g_N(w)=\min\bigl(1,\;N\operatorname{dist}
\bigl(w,\,[0,\tfrac14]\setminus(a,b)\bigr)\bigr),
\qquad w\in[0,\tfrac14].
\]
Then $g_N$ is continuous, $0\le g_N\le\ind{(a,b)}$, $g_N$ is nondecreasing in
$N$, and $g_N\uparrow\ind{(a,b)}$ pointwise on $[0,\frac14]$: for $w\in(a,b)$
the distance is $\min(w-a,b-w)>0$, so $N$ times it exceeds $1$ once
$N>\bigl(\min(w-a,b-w)\bigr)^{-1}$, while for $w\notin(a,b)$ the distance and
hence $g_N(w)$ vanish for every $N$. What the next step needs is that
$g_N\equiv0$ on $[0,a]$ with $a>0$.

\emph{Step 2 (dividing by $w$).} Put $h_N(w)=g_N(w)/w$ for $w>0$ and
$h_N(0)=0$. Since $g_N$ vanishes identically on $[0,a]$ with $a>0$, $h_N$
vanishes identically on $[0,a]$ and equals a quotient of continuous functions
with nonvanishing denominator on $[a,\frac14]$; so $h_N\in C[0,\frac14]$.

\emph{Step 3 (Weierstrass, then a one-sided margin).} By the Weierstrass
approximation theorem choose a real polynomial $q_N$ with
$\sup_{[0,1/4]}|q_N-h_N|\le\eta_N:=\frac1N$, and set
\[
\varrho_N(w)\;:=\;w\bigl(q_N(w)-\eta_N\bigr).
\]
Then $\varrho_N(0)=0$, which is (A1). For (A2): for every $w\in[0,\frac14]$
we have $w\ge0$ and $q_N(w)-\eta_N\le h_N(w)$, hence
\[
\varrho_N(w)\ \le\ w\,h_N(w)\ =\ g_N(w)\ \le\ \ind{(a,b)}(w).
\]
This is where the margin $\eta_N$ is spent, and it is spent on the correct
side: subtracting it converts a two-sided approximation into a one-sided
minorant. Nothing forces $\varrho_N\ge0$ anywhere, and in general it is
negative somewhere.

\emph{Step 4 (the coefficient converges).} By \eqref{eq:Ifunctional},
\[
\mathcal I[\varrho_N]
=\frac1{\pi^2}\int_0^1\bigl(q_N(w)-\eta_N\bigr)\,d\lambda,
\qquad
\frac1{\pi^2}\int_0^1h_N(w)\,d\lambda
=\frac1{\pi^2}\int_0^1g_N(w)\frac{d\lambda}{w},
\]
both integrals being over $\lambda\in(0,1)$ with $w=\lambda(1-\lambda)$; the
first is finite because $q_N$ is a polynomial, and $|q_N-h_N|\le\eta_N$
gives
\[
\Bigl|\mathcal I[\varrho_N]
-\frac1{\pi^2}\int_0^1g_N(w)\frac{d\lambda}{w}\Bigr|
\ \le\ \frac{2\eta_N}{\pi^2}\ =\ \frac{2}{\pi^2N}\ \longrightarrow\ 0 .
\]
By Step~1, $g_N(w)/w\uparrow\ind{(a,b)}(w)/w$ pointwise on $(0,1)$ and all
terms are nonnegative, so monotone convergence gives
$\frac1{\pi^2}\int_0^1g_N(w)\frac{d\lambda}{w}\to\mathcal I^\ast(a,b)$.
Hence $\mathcal I[\varrho_N]\to\mathcal I^\ast(a,b)$.

The final sentence combines (i) and (ii).
\end{proof}

\begin{remark}[endpoints]\label{rem:minorantendpoints}
Lemma~\ref{lem:minorant} is stated with the \emph{open} window $(a,b)$ and
delivers \eqref{eq:minorantcount} with $\#\{n:\varphi_n\in(a,b)\}$ on the
right, and this is what Proposition~\ref{prop:fixeddepth} consumes: the
preimage of the open $w$-interval under $\lambda\mapsto\lambda(1-\lambda)$ is
the union of two \emph{open} $\lambda$-intervals, which is exactly
$W_0\cup W_1$. No assumption is made that $\varphi$ has no eigenvalue at $a$
or at $b$; eigenvalues sitting at an endpoint are simply not counted on
either side, and the inequality is unaffected. When $b=\frac14$ the same
proof applies with $\ind{(a,1/4]}$ in place of $\ind{(a,b)}$, because
$\varphi_n\le\frac14$ always makes the two indicators agree at every point of
the spectrum; that is the case of Proposition~\ref{prop:count}, where the
minorant is displayed rather than constructed.
\end{remark}

\begin{proposition}[Fixed-depth density]\label{prop:fixeddepth}
Assume Theorem~\ref{thm:basorwidom}. Fix $u>\ln2$ and set
\[
a_u=e^{-2u}(1-e^{-2u}),\qquad b_u=e^{-u}(1-e^{-u}),
\qquad\text{so that }0<a_u<b_u<\tfrac14 .
\]
Then for every $c_1$ with
$0<c_1<\frac2{\pi^2}\bigl(u+\ln(1+e^{-u})\bigr)$ there is a finite $C_1$,
depending on $u$ and on $c_1$, with
\[
W_0(u,\ell)+W_1(u,\ell)\ \ge\ c_1\ln\ell-C_1
\qquad(\ell\ge1).
\]
In particular the plunge has density $\gtrsim\frac1{\pi^2}$ per unit
of depth on any compact depth range in $(\ln2,\infty)$.
\end{proposition}

\noindent
Three features of the statement deserve comment before the proof.

\emph{It is two-sided, and it is not a lower-half statement.} What is bounded
below is the symmetric pair $W_0+W_1$, a tail window at depths in $(u,2u)$
together with its reflected head window at depths in $(-2u,-u)$, and not
$W_0$ alone. The method forces this rather than choosing it: the moment
inequality acts on $\varphi=S_\ell-S_\ell^2$, whose spectrum is
$\lambda_n(1-\lambda_n)$, and $\lambda\mapsto\lambda(1-\lambda)$ identifies
each $\lambda$ with $1-\lambda$, so a $w$-window always pulls back to a pair
of $\lambda$-windows symmetric about $\frac12$. Splitting the pair
(bounding the lower-half count alone) is not achieved here.

\emph{The ceiling is exact but not attained.} The supremum
$\frac2{\pi^2}(u+\ln(1+e^{-u}))$ is the exact value of
$\mathcal I^\ast(a_u,b_u)$, and no polynomial attains it; it is approached,
by Lemma~\ref{lem:minorant}(ii), at the cost of a larger degree $m_0$.

\emph{Nothing here is uniform in $u$.} The constant $C_1$ is finite but not
effective, and \emph{no estimate uniform as $u\to\infty$ is claimed}: the
construction supplies a finite degree at each fixed depth with no degree
estimate uniform in the depth, and the fixed-order trace constants of
Proposition~\ref{lem:higher} deteriorate with that degree
(Remark~\ref{rem:mquant}, Section~\ref{sec:stops}).

\begin{proof}
\emph{Step 0 (the window in $w$).} Since $u>\ln2$ we have $e^{-u}<\frac12$,
so $t\mapsto t(1-t)$ is increasing on $(0,e^{-u}]$ and
$0<a_u<b_u<\frac14$, the last inequality strict because $e^{-u}\ne\frac12$.
The preimage of $(a_u,b_u)$ under $\lambda\mapsto\lambda(1-\lambda)$, inside
$(0,1)$, is
$(e^{-2u},e^{-u})\cup(1-e^{-u},1-e^{-2u})$, which is exactly $W_0$'s window
together with $W_1$'s. The two are disjoint because $u>\ln2$ places
$e^{-u}$ below $\frac12$ and $1-e^{-u}$ above it. Hence
$\#\{n:\varphi_n\in(a_u,b_u)\}=W_0(u,\ell)+W_1(u,\ell)$.

\emph{Step 1 (the value of $\mathcal I^\ast$).} With
$[\ln\frac\lambda{1-\lambda}]$ as antiderivative of
$\frac{d\lambda}{\lambda(1-\lambda)}$ and the substitution
$\lambda\mapsto1-\lambda$ pairing the two components,
\[
\mathcal I^\ast(a_u,b_u)
=\frac2{\pi^2}\Bigl[\ln\frac\lambda{1-\lambda}\Bigr]_{e^{-2u}}^{e^{-u}}
=\frac2{\pi^2}\ln\frac{e^{-u}(1-e^{-2u})}{e^{-2u}(1-e^{-u})}
=\frac2{\pi^2}\bigl(u+\ln(1+e^{-u})\bigr),
\]
using $1-e^{-2u}=(1-e^{-u})(1+e^{-u})$.

\emph{Step 2 (choosing the polynomial).} Given $c_1<\mathcal I^\ast(a_u,b_u)$,
Lemma~\ref{lem:minorant}(ii) supplies a polynomial $\varrho$ admissible for
$(a_u,b_u)$, of some finite degree $m_0$, with $\mathcal I[\varrho]\ge c_1$.
Fix it. By \eqref{eq:minorantcount} and Step~0,
\[
W_0(u,\ell)+W_1(u,\ell)\ \ge\ \sum_{m=1}^{m_0}b_m\Tr(\varphi^m).
\]

\emph{Step 3 (evaluating the traces).} The polynomial $\varrho$ has fixed
finite degree
$m_0$, so the bound of Step~2 is a fixed finite linear combination
$\sum_{m\le m_0}b_m\Tr(\varphi^m)$ of traces at \emph{finitely many fixed}
orders $m$. Proposition~\ref{lem:higher} supplies each of those, at that
fixed $m$, as $\beta_m\pi^{-2}\ln\ell+O_m(1)$ with a remainder bounded
uniformly in $\ell\ge1$. Summing with the fixed coefficients $b_m$, the
$\ln\ell$ terms assemble to
$\bigl(\pi^{-2}\sum_mb_m\beta_m\bigr)\ln\ell=\mathcal I[\varrho]\ln\ell\ge
c_1\ln\ell$ by \eqref{eq:Ifunctional} and Step~2, and the $m_0$ remainders
assemble to a finite constant, which we call $C_1$.
The order of the quantifiers is what makes this work: \emph{first} $u$ and
$c_1$ are fixed, \emph{then} $\varrho$ and hence $m_0$, \emph{then}
Proposition~\ref{lem:higher} is invoked at each of the finitely many
$m\le m_0$. At no point is a remainder uniform in $m$ required, and none is
available (Remark~\ref{rem:mquant}).

That same order is why nothing is gained as $u$ varies: both $C_1$ and the
$\ell$-threshold behind it are controlled only at the fixed $u$ chosen at the
outset, and the constant in Proposition~\ref{lem:higher} grows like
$(16/3)^{2m}$ with the degree. The proposition is a family of statements
indexed by $u$, each proved, with no uniformity across the family. $C_1$ is
finite but not effective for the same reason $C_0$ is in
Theorem~\ref{thm:window}: the threshold $\ell_1$ inside
Proposition~\ref{lem:higher} rests on an implied constant the cited source
does not quantify.
\end{proof}

\begin{remark}[an explicit minorant, and what it costs]\label{rem:explicitmin}
Lemma~\ref{lem:minorant}(ii) is an existence statement, and one admissible
polynomial can be displayed outright. For $0<a<b<\frac14$ and an
integer $r\ge1$ put
\[
\varrho_r(w)=\kappa_r\,w^{\,r}(w-a)(b-w),\qquad
\kappa_r=\Bigl(\max_{a\le w\le b}w^{\,r}(w-a)(b-w)\Bigr)^{-1}.
\]
Admissibility is immediate: $\varrho_r(0)=0$; on $[0,a]$ the factor $(w-a)$
is $\le0$ and $(b-w)>0$, on $[b,\frac14]$ the signs are reversed, so
$\varrho_r\le0$ off $(a,b)$; and $\varrho_r\le1$ on $[a,b]$ by the choice of
$\kappa_r$. Its coefficient is closed-form: with
$t_r=\beta_{r+1}/\beta_r=\frac r{2(2r+1)}$,
\[
\mathcal I[\varrho_r]=\frac{\kappa_r\beta_r}{\pi^2}
\Bigl[(t_r-a)(b-t_r)-\frac{t_r}{2(2r+1)(2r+3)}\Bigr],
\]
and since $\max_{[a,b]}w^r(w-a)(b-w)\le b^{\,r}(b-a)^2/4$, the fully explicit
bound $\kappa_r\ge4b^{-r}(b-a)^{-2}$ may be substituted whenever the bracket
is positive. The bracket is positive only when $t_r$ falls inside $(a,b)$
with room to spare, which confines this one-bump family to a bounded range of
depths; past that range one must use Lemma~\ref{lem:minorant}(ii) itself,
whose polynomials are not single-bump. That is the cost of insisting on an
explicit minorant, and the reason the lemma is stated as an existence result.
\end{remark}

\noindent
Each instance of Proposition~\ref{prop:fixeddepth} consumes finitely many
instances of Proposition~\ref{lem:higher} and exactly one instance of
Lemma~\ref{lem:minorant}. Both are proved above, which is what leaves the
proposition free of any hypothesis on the eigenvalue profile. No symmetry
between $W_0$ and $W_1$ is proved anywhere below, and none is used.

\subsection{\texorpdfstring{The corner block in every dimension: proof of
Theorem~\ref{thm:lower}}{The corner block in every dimension: proof of
Theorem 1.4}}\label{sec:corner}

The one-dimensional count of Theorem~\ref{thm:window} now produces a
$d$-dimensional block, by the tensor identity read as a counting statement.

\begin{proof}[Proof of Theorem~\ref{thm:lower}]
Let $d\ge2$, $A_0=B_0=(0,1)^d$, $c=\ell$, and $\eps<4^{-d}$.

\emph{Step 1 (the product window).} If $\lambda_{n_1},\dots,\lambda_{n_d}$
all lie in $(\frac14,\frac34)$, then
\[
\prod_{i=1}^d\lambda_{n_i}\in\bigl(4^{-d},(\tfrac34)^d\bigr),
\]
the endpoints being the products of the endpoints and both inclusions
strict.

\emph{Step 2 (the window sits inside the plunge).} On the lower side,
$4^{-d}>\eps$ is the hypothesis. On the upper side we need
$(\frac34)^d<1-\eps$, and since $\eps<4^{-d}$ it suffices that
$1-(\frac34)^d-4^{-d}>0$; this holds for every $d\ge2$, the tightest case
being $d=2$, where $1-\frac9{16}-\frac1{16}=\frac38>0$. Hence
$\bigl(4^{-d},(\frac34)^d\bigr)\subset(\eps,1-\eps)$.

\emph{Step 3 (counting the block).} By \eqref{eq:prodcount} every one of the
$M_a^{\,d}$ multi-indices $(n_1,\dots,n_d)$ with all $d$ entries drawn from
the $M_a$ indices of \eqref{eq:Ma} is counted by $\Lambda_\eps(c)$; distinct
multi-indices are distinct eigenvalue slots, so
$\Lambda_\eps(c)\ge M_a^{\,d}$. Note that Steps 1--2 give membership in the
\emph{open} interval $(\eps,1-\eps)$, so this is the count of
Definition~\ref{def:plunge} and not the larger half-open count of
Remark~\ref{rem:halfopen}.

\emph{Step 4 (inserting the one-dimensional count).} $M_a$ is a nonnegative
integer, so Theorem~\ref{thm:window} gives
$M_a\ge(c_0\ln\ell-C_0)_+$; and $t\mapsto t^{\,d}$ is increasing on
$[0,\infty)$, so
\[
\Lambda_\eps(c)\ \ge\ M_a^{\,d}\ \ge\
\bigl((c_0\ln\ell-C_0)_+\bigr)^{d}\ =\ \Omega\bigl((\log c)^d\bigr),
\]
the last equality because $c=\ell$ and $c_0>0$. \emph{The positive part is
needed:} for $d$ even, $x\mapsto x^d$ is not monotone on all of $\R$, and at
small $\ell$ the bracket $c_0\ln\ell-C_0$ is negative while $M_a$ may be $0$,
so the inequality would fail if the bracket were raised to an even power as
it stands.
\end{proof}

\begin{remark}[what changes with the dimension]\label{rem:lift1}
At $d=2$ the threshold $\eps<4^{-d}$ reads $\eps<\frac1{16}$, and Steps
1--4 are the $d=2$ argument of this section, unchanged; the statement for
$d\ge3$ is proved here. Two things change with $d$, and both are visible in
the statement: the $\eps$-threshold $4^{-d}$ contracts geometrically, so
Theorem~\ref{thm:lower} says nothing at a fixed $\eps$ once
$d>\log_4(1/\eps)$; and the exponent on the logarithm rises with $d$, which
is the assertion. The tensor identity itself (Lemma~\ref{lem:tensor},
Remark~\ref{rem:cube}) is dimension-free and needs no extension.
\end{remark}

\subsection{An unconditional edge law}\label{sec:edge}

The window count also yields a lower bound on the
Schatten mass $N_q(\ell,b)=\sum_n\lambda_n^{\,q}$ that assumes nothing about
the eigenvalue profile. A bound of this
shape is otherwise available only conditionally on Landau--Widom
asymptotics; Theorem~\ref{thm:window} trades that conditioning for a
citation, the determinant asymptotics of Theorem~\ref{thm:basorwidom}, which
is a published theorem rather than an assumed profile. That is a change in
the \emph{kind} of input and not in its effectiveness: both routes leave the
threshold beyond which the bound has content unquantified, here through
$C_0$ (Section~\ref{sec:tier1}). We record it here because later steps of the
program consume it.

\begin{theorem}[Unconditional $N_q$]\label{thm:nq}
For $0<q\le\frac12$ and $b=1$,
\begin{equation}\label{eq:nq}
N_q(\ell,b)\ \ge\ \ell b+\delta_q\bigl(c_0\ln\ell-C_0\bigr),
\qquad \delta_q=\bigl(\tfrac34\bigr)^q-\tfrac34\ \ge\ 0.116,
\end{equation}
with $c_0,C_0$ as in Theorem~\ref{thm:window}. In particular
$N_q(\ell,b)-\ell b\gtrsim\ln(\ell b)$ with a constant independent of
$q\in(0,\frac12]$; the constant in \eqref{eq:nq} is
$c_3=\delta_qc_0\ge\delta_{1/2}c_0=\frac{4\sqrt3-6}{15\pi^2}$.
\end{theorem}

\begin{proof}
$N_q-\ell b=\sum_n(\lambda_n^{\,q}-\lambda_n)$ since $\Tr S_\ell=\ell b$,
and each term is $\ge0$ for $q\le1$. Restricting the sum to
$\lambda_n\in(\frac14,\frac34)$ and using that
$f(\lambda)=\lambda^q-\lambda$ satisfies
$f\ge\min_{[1/4,3/4]}f=f(\frac34)=(\frac34)^q-\frac34=\delta_q$ there
(the minimum is at the right endpoint, since $f$ decreases past its interior
maximum $\lambda_\ast=q^{1/(1-q)}\le\frac14$ for $q\le\frac12$),
\[
N_q-\ell b\ \ge\!\!\sum_{\lambda_n\in(1/4,3/4)}\!\!(\lambda_n^{\,q}-\lambda_n)
\ \ge\ \delta_q\,M_a\ \ge\ \delta_q(c_0\ln\ell-C_0).
\]
Finally $\delta_q$ decreases in $q$ on $(0,\frac12]$, from $\frac14$ toward
$(\frac34)^{1/2}-\frac34=0.116\ldots$, giving the uniform constant.
\end{proof}

No positive part is needed in \eqref{eq:nq}: the bound is affine in the
bracket, and $N_q\ge\ell b$ holds outright. The estimate is deliberately
crude, since the proof retains only the eigenvalues in the fixed window
$(\frac14,\frac34)$ and discards the rest of the excess.

\subsection{Where the moment method stops}\label{sec:stops}

Everything above is proved, given the citation of
Theorem~\ref{thm:basorwidom}, and everything above is confined to
\emph{fixed} depth. The confinement is structural rather than a matter of
presentation. Lemma~\ref{lem:minorant}(ii) supplies a finite degree for every
fixed depth $u$ but no degree estimate uniform as $u$ tends to infinity, and
the fixed-order trace constants of Proposition~\ref{lem:higher} deteriorate
with the degree. The argument therefore yields no moving-depth estimate, and
in particular does not reach $u\to L=\ln\frac1\eps$ as $\eps\downarrow0$.
That is the quantifier structure recorded at
Proposition~\ref{prop:fixeddepth} and in Remark~\ref{rem:mquant}: each
instance is proved, none is uniform. No rate is asserted, proved or used for
the growth of the required degree with $u$.

This is the boundary of what the present method reaches, and what lies past
it should be named precisely.

A lower bound matching the $c^{\,d-1}LR$ order of
Corollary~\ref{cor:kdl} would require one-dimensional lower-half density
estimates in a depth range that grows with $\ln\frac1\eps$. The fixed-window
argument used here does not reach such a range, for the reason just given.

Estimates of that kind are not absent from the literature. For the cube pair,
\cite[Thm.~1.1]{KDL} bounds $\Lambda^-_\eps$ from below by a constant
multiple of
$c^{\,d-1}\log\frac1\eps\,\log\bigl(\alpha_dc/\log\frac1\eps\bigr)$ on the
deep polynomial range $\alpha_d^{-c}<\eps<c^{-\alpha_d}$, and determines its
order outright for $\eps\le\alpha_d^{-c}$. What is missing, as far as we
know, is an estimate of that shape holding \emph{uniformly} for $\eps$ moving
between a fixed threshold and that deep regime. Obtaining one by the
fixed-window trace method of this section is a separate problem, since that
method is confined to fixed depth by the degree growth just described.
Nothing above depends on it.

\begin{remark}[what is, and is not, uniform in $m$]\label{rem:mquant}
Two limitations of Proposition~\ref{lem:higher} bear on what is claimed
later, and it matters exactly where they bite.

\emph{The bound is not uniform in $m$.} The remainder prefactor in Step~2
carries $(16/3)^{2m}$, inherited from Cauchy's estimate on the circle
$|\sigma|=\frac3{16}$ in Lemma~\ref{lem:extract}. Nothing in
\eqref{eq:higherm} therefore survives $m$ growing with $\ell$.
Proposition~\ref{lem:higher} asserts a family of statements, one for each
fixed $m$; it is not a statement about the family, and no result below reads
it as one.

\emph{The $O_m(1)$ is finite but not effective.} The threshold $\ell_1$ rests
on the implied constant in the $O(\ln s/s)$ of
Theorem~\ref{thm:basorwidom}, which \cite{Charlier21} states to be uniform
but does not quantify. This is the same non-effectivity that $C_0$ carries in
Theorem~\ref{thm:window}, and from the same source. The leading coefficient
$\beta_m/\pi^2$ is exact and is unaffected.
\end{remark}

\section{What breaks off the box class}\label{sec:limits}

What was used and what was not is set out here, since the purpose of this
step of the program is to expose the mechanism for the general case.

\begin{enumerate}[label=(\alph*),itemsep=2pt]
\item\emph{Where the product structure entered.} Twice, and only twice: the
telescoping \eqref{eq:telescope} of $\ind{R^c}$, which needs $R$ to be a
box, and the factorization $Q_{B^\circ}=\bigotimes_mQ_{B^\circ_m}$, which
needs $B^\circ$ to be a box. Everything else is geometry-free: the plunge
reduction, Rotfel'd assembly, the modulation and translation normalizations,
tensor multiplicativity, the tangential mass lemma, and the one-dimensional
input.
\item\emph{Slanted parallelepipeds.} If $V\in GL(d,\R)$ and
$f\mapsto|\det V|^{1/2}f(Vx)$ is the associated unitary, it conjugates $P_E$
to $P_{V^{-1}E}$ and $Q_F$ to $Q_{V^{\mathsf T}F}$. Hence
Theorem~\ref{thm:upper} holds verbatim when every $R^{(i)}$ is a translate
of a $V$-image of an axis box and every $Q^{(j)}$ a translate of a
$V^{-\mathsf T}$-image of an axis box, for one common $V$, that is, for a
\emph{common-orientation} pair of parallelepiped families. Mixed
orientations are not covered: a single linear change cannot simultaneously
rectify boxes of different orientations, and after rectifying $A$ the
frequency boxes become slanted, destroying the tensorization of $Q_B$.
\item\emph{One general set.} \cite[Thm.~1.3]{KDL} allows
one of the two sets to be arbitrary bounded with finite upper Minkowski
boundary content, provided the other is a finite union of parallelepipeds
with disjoint interiors. Our method does not reproduce this: a non-product
$B$ breaks \eqref{eq:tensor} even for $A$ a box. \emph{This is the precise
sense in which \cite{KDL} is stronger than the present paper on the upper
side}, and the box-class corollary we do prove is contained in theirs.
At \emph{fixed} $\eps$ a single logarithm in $c$ is in fact already available
with no box on either side, by the trace route of \cite{HIMrough} applied to
any two sets of finite perimeter, at the cost of a factor $\eps^{-1}$ in
place of $L$; what \eqref{eq:conj} asks for beyond this is the joint
dependence on $c$ and $\eps$.
Recovering the general one-sided statement, and with it the full conjecture,
is the task of the subsequent steps, whose plan is: a
partition of $\partial(cA_0)$ into patches of tangential diameter
$\asymp\sqrt c$ (the scale at which the quadratic deviation of a $C^{1,1}$
boundary from its tangent plane produces only an $O(1)$ phase across the
frequency support); a straightening of each patch, reducing $T$-pieces to
perturbed copies of the flat tensor model \eqref{eq:tensor}; and a
frequency-side patching of $\partial B_0$ producing, per patch pair, an
effective one-dimensional band in the normal direction. This paper supplies
the flat model those reductions must land on; the curvature and cross-patch
error estimates in Schatten $p<1$ are the open work, and are not claimed
here.
\item\emph{Sharpness of the tangential input.} The leading constant $2$ in
Lemma~\ref{lem:mass}, from Markov at level $\frac12$, is the only place the
head of the tangential spectrum is touched; replacing it by the
asymptotically correct $1+o(1)$ would improve the constant in
Theorem~\ref{thm:upper} by $2^{d-1}$, but is irrelevant to the conjecture.
\item\emph{Consistency.} For $d=1$, Theorem~\ref{thm:upper} has empty
products and reduces to the (sharper, intermediate) per-component form of
the main theorem of \cite{Afard1d}, and hence implies it;
for fixed $\eps$ and a single pair of boxes it gives
$\Lambda_\eps=O(c^{d-1}\log c)$, the multidimensional Landau--Widom order
\cite{LandauWidom}.
\end{enumerate}

\subsection{The scope of the lower bounds}\label{sec:scopelower}

Every lower bound in Section~\ref{sec:lower} is proved, from
Theorem~\ref{thm:basorwidom} and nothing else. What is proved and what is
claimed are not the same thing, and the two are separated here.

\emph{What is established.} The window count (Theorem~\ref{thm:window}), the
higher traces at every fixed order (Proposition~\ref{lem:higher}), the
fixed-depth density (Proposition~\ref{prop:fixeddepth}), the logarithmic
tensor block in every dimension (Theorem~\ref{thm:lower}) and the trace
estimate (Theorem~\ref{thm:nq}) assume nothing about the eigenvalue profile.
Their one external input is Theorem~\ref{thm:basorwidom}, a published
theorem, reached through Appendix~\ref{app:exact}. What
Theorem~\ref{thm:lower} delivers is a genuine $d$-fold product block of
plunge eigenvalues, of size $\Omega((\log c)^d)$.

\emph{What is not claimed: a matching order.} $\Omega((\log c)^d)$ is not a
matching lower bound of order $c^{\,d-1}LR$, and this paper does not prove
one. The two differ by a power of $c$.

\emph{What is not claimed: uniformity.} Proposition~\ref{lem:higher} is a
statement at each fixed order $m$, with constants that degrade in $m$, and
Proposition~\ref{prop:fixeddepth} is a statement at each fixed depth $u$,
with constants that degrade in $u$ (Remark~\ref{rem:mquant},
Section~\ref{sec:stops}). Neither is uniform over its family, and neither is
used here as though it were; $C_0$ of Theorem~\ref{thm:window} and $C_1(u)$
of Proposition~\ref{prop:fixeddepth} inherit the non-effectivity described
there. The leading coefficients $c_0$ and $\beta_m/\pi^2$ are exact
and explicit, and the ceiling for $c_1$ is exact; $c_0$ is optimal at
degree three (Lemma~\ref{lem:degree3}) and is \emph{not} the all-degree
ceiling for its window (Remark~\ref{rem:c0status}).

\emph{What is not claimed: priority for the upper bound's order.} The
$c^{\,d-1}LR$ order of Corollary~\ref{cor:kdl} is already established, for a
strictly larger geometric class and on the same range, by
\cite[Thm.~1.3]{KDL}; see Section~\ref{sec:priorwork}. The contribution here
is Theorem~\ref{thm:upper}'s all-parameter explicit form, together with the
pair reduction and telescoping tensorization of the off-diagonal factor that
yield this form for finite unions of boxes.

\subsection{The remaining steps}\label{sec:signposts}

The mechanism this paper isolates, exact tensor structure standing in for
almost-orthogonality, is what the remaining steps of the program deform.
The first carries the argument to curved spatial boundaries with the
frequency set still a box; the second does the reverse, curved frequency
sets over box spatial sets; the third treats the curved pair, where both
patchings run at once; and the fourth assembles the recursive skeletons that
the patch decompositions of the first three generate.

On the lower side, one direction is named here but not pursued.
\cite[Thm.~1.1]{KDL} supplies lower-half estimates of the matching order in a
deep polynomial regime for the cube pair, and the exact order below it.
Obtaining a comparable bridge from fixed depth to that regime, by the trace
method of Section~\ref{sec:stops} or otherwise, is a separate problem, and
nothing in this paper depends on it.

\appendix

\section{\texorpdfstring{Ky Fan and Rotfel'd inequalities, as
used}{Ky Fan and Rotfeld inequalities, as used}}\label{app:kyfan}

Only two facts about singular values of sums are used, and only in the
quasi-Banach range. We state them in the form in which they are applied and
record where.

\begin{theorem}[Rotfel'd \cite{Rotfeld}; cf.\ {\cite[Thm.~2.8]{SimonTI}}]
\label{thm:rotfeld}
Let $X,Y$ be compact operators between separable Hilbert spaces and let
$f:[0,\infty)\to[0,\infty)$ be concave and non-decreasing with $f(0)=0$.
Then
\[
\sum_{n\ge1}f\bigl(s_n(X+Y)\bigr)\ \le\
\sum_{n\ge1}f\bigl(s_n(X)\bigr)+\sum_{n\ge1}f\bigl(s_n(Y)\bigr),
\]
both sides being allowed the value $+\infty$.
\end{theorem}

\begin{corollary}[quasi-norm subadditivity]\label{cor:rotfeld}
For $0<p\le1$ and compact $X_1,\dots,X_r$,
\[
\Bigl\|\sum_{k=1}^rX_k\Bigr\|_p^p\ \le\ \sum_{k=1}^r\|X_k\|_p^p .
\]
\end{corollary}

\begin{proof}
$f(t)=t^{\,p}$ is concave and increasing on $[0,\infty)$ with $f(0)=0$ for
$0<p\le1$; apply Theorem~\ref{thm:rotfeld} and induct on $r$.
\end{proof}

This is Lemma~\ref{lem:svcalc}(ii), and it is used three times: over the $MK$
component pairs in Proposition~\ref{prop:pairs}, over the $d$ directions of
the telescoping \eqref{eq:telescope} in Proposition~\ref{prop:tensor}, and,
inside the one-dimensional input \cite{Afard1d} rather than here, over the
dyadic normal scales. The first two uses cost an additive term and no
logarithm, since the sums they generate run over $MK$ and over $d$, both
independent of $\eps$ and of $c$. The third is where the logarithm enters:
the dyadic normal scales number $O(\log)$, and that count is the
$4.5\logp(\ell p)$ term of Proposition~\ref{prop:1d}. This is what stands
behind the claim of Section~\ref{sec:mech}: subadditivity is invoked over
$\eps$-independent index sets everywhere \emph{except} once, inside the
one-dimensional input, and that single exception is the single logarithm.

The second fact is Ky Fan's inequality
$s_{m+n-1}(X+Y)\le s_m(X)+s_n(Y)$, of which we use only the case
$m=n=1$, $\|X+Y\|\le\|X\|+\|Y\|$, and the consequence recorded as
Lemma~\ref{lem:svcalc}(i), $s_n(UXV)\le\|U\|\,\|V\|\,s_n(X)$. Both are
standard; no other rearrangement inequality enters.

\begin{remark}[why subadditivity and not orthogonality]
Corollary~\ref{cor:rotfeld} is the \emph{only} general inequality available
for sums in the range $0<p<1$: the triangle inequality fails, and there is
no Cotlar--Stein substitute, because independent pieces genuinely add
Schatten mass (Section~\ref{sec:mech}). This paper does not improve on it;
it arranges the decomposition so that subadditivity is invoked $O(d\cdot MK)$
times rather than $O(\log)$ times.
\end{remark}

\section{Explicit constants}\label{app:constants}

The auxiliary constants in Corollaries~\ref{cor:kdl} and~\ref{cor:matched}
follow from the one-sided elementary inequalities displayed in their proofs.
Each decimal approximation is rounded in the direction required by the
corresponding estimate. No optimization of these constants is claimed.

\section{\texorpdfstring{The trace machinery, and the $m=2,3$ constants
exactly}{The trace machinery, and the m = 2, 3 constants
exactly}}\label{app:exact}

This appendix supplies the analytic input to Section~\ref{sec:lower}, and
what it supplies is stated at the outset, since the paper's claims are
calibrated to it.

\emph{What is proved here.} Lemma~\ref{lem:extract} gives
$\Tr(S_\ell^{\,k})$ at \emph{every} fixed $k\ge1$, with an explicit
remainder. Lemma~\ref{lem:harmonic} evaluates the alternating harmonic sum
that converts those into the Beta values $\beta_m=B(m,m)$; it is an exact
combinatorial identity with no analytic content. Together they prove the
asymptotics of $\Tr(\varphi^m)$ at every fixed $m$, which is
Proposition~\ref{lem:higher}, since $\Tr(\varphi^m)$ is the finite combination
$\sum_{j=0}^m(-1)^j\binom mj\Tr(S_\ell^{\,m+j})$ of the former.
Lemma~\ref{lem:exacttrace} then does more at the two orders the window count
consumes, $m=2$ and $m=3$: it evaluates the constant terms in closed form and
sharpens the remainder to a power saving.

\emph{What is not proved here.} The remainder bound
$k(\tfrac{16}3)^k\delta_\ell^{\,\vartheta}\mathsf M_\ell^{1-\vartheta}$ of
Lemma~\ref{lem:extract} degrades geometrically in $k$, so nothing in this
appendix survives $m$ growing with $\ell$. Every statement it feeds is
therefore a statement at fixed order. The results that consume it,
Proposition~\ref{lem:higher} and, through it,
Proposition~\ref{prop:fixeddepth}, are stated at fixed $m$ and at fixed
depth $u$ respectively, with constants allowed to depend on them
(Remark~\ref{rem:mquant}).

\medskip\noindent\textbf{Notation for this appendix.} Sans-serif letters
denote constants and functions of the determinant expansion and are distinct
from every symbol of Sections~\ref{sec:intro}--\ref{sec:limits}: in
particular $\mathsf{A}$ is a number, not the spatial set $A$; $\mathsf{B}$ is
a number, not the frequency set $B$ nor the Beta function $B(m,m)$ of
Proposition~\ref{lem:higher}; $\mathsf{G}$ is the Barnes $G$-function, not the
normal profile $G$ of \eqref{eq:G}; $\mathsf{H}_k$ is a harmonic number, not
the tensor factor $H^{(ij)}_m$ of \eqref{eq:GH}; and $\mathsf{c}_k$,
$\mathsf{K}_2$, $\mathsf{K}_3$ are not $c_0$, $c_1$, $c_3$ or the
component count $K$. The Joukowski radii $R$, $R_0$, $R_2$ of
Lemma~\ref{lem:extract} are likewise unrelated to the boxes $R^{(i)}$ of
Assumption~\ref{ass:boxes}. The determinant parameter is written $\sigma$,
the exponent produced by the two-constants theorem $\vartheta$, and $\zeta$
is the Riemann zeta function; none of these three occurs elsewhere in the
paper. Throughout, $S_\ell$ and $\varphi=S_\ell-S_\ell^2$ are as in
Section~\ref{sec:1d}, and
\[
\mathsf{A}:=1+\gamma_{\mathrm E}+\ln2\pi,
\qquad
\mathsf{H}_k:=1+\tfrac12+\cdots+\tfrac1k,\quad \mathsf{H}_0:=0 .
\]

\subsection{The cited input}\label{app:cited}

\begin{theorem}[Basor--Widom; as reproved by Charlier]\label{thm:basorwidom}
Let $K_s$ be the sine kernel $\sin s(x-y)/(\pi(x-y))$ on $(-1,1)$. For fixed
real $\sigma<1$, as $s\to+\infty$,
\[
\ln\det(I-\sigma K_s)=\frac{2s}\pi\ln(1-\sigma)
+\frac{\ln^2(1-\sigma)}{2\pi^2}\ln(4s)
+2\ln\bigl[\mathsf{G}(1+i\mathsf{z})\mathsf{G}(1-i\mathsf{z})\bigr]
+O\!\Bigl(\frac{\ln s}{s}\Bigr),
\]
where $\mathsf{G}$ is the Barnes $G$-function and
$\mathsf{z}=\mathsf{z}(\sigma)=-\ln(1-\sigma)/(2\pi)$. The error term is
uniform for $\sigma$ in compact subsets of $(-\infty,1)$.
\end{theorem}

This is \cite[Thm.~1.1]{Charlier21} specialized to $m=1$, originally
\cite{BasorWidom83}. We quote it rather than reprove it. Its two-sidedness in
$\sigma$, over a segment symmetric about $\sigma=0$, is what makes
$\vartheta$ below a closed-form number rather than an unevaluated constant.
Rescaling $x\mapsto x/s$ and then by $\pi$ carries $K_s$ to $\sinc$ on an
interval of length $2s/\pi$, so $S_\ell\cong K_s$ unitarily with
$s=\pi\ell/2$; under that dictionary Theorem~\ref{thm:basorwidom} reads
\begin{equation}\label{eq:detasym}
\ln\det(I-\sigma S_\ell)=\mathsf{P}_\ell(\sigma)+\mathsf{E}_\ell(\sigma),
\qquad
\mathsf{P}_\ell(\sigma):=\ell\ln(1-\sigma)
+\frac{\ln^2(1-\sigma)}{2\pi^2}\ln\ell+\mathsf{C}(\sigma),
\end{equation}
with
$\mathsf{C}(\sigma):=2\ln[\mathsf{G}(1+i\mathsf{z})\mathsf{G}(1-i\mathsf{z})]
+\frac{\ln^2(1-\sigma)}{2\pi^2}\ln2\pi$, and the theorem asserts exactly
$\mathsf{E}_\ell(\sigma)=O(\ln\ell/\ell)$, uniformly on compact real
$\sigma$-sets.

\subsection{Three lemmas, all proved here}\label{app:twolemmas}

\begin{lemma}[crude bound on the whole disc]\label{lem:crude}
Let $0<\rho<1$. For every $\ell\ge1$ and every complex $\sigma$ with
$|\sigma|\le\rho$,
\[
\bigl|\ln\det(I-\sigma S_\ell)-\ell\ln(1-\sigma)\bigr|
\ \le\ \frac{\rho^2}{1-\rho}\,\Tr(S_\ell-S_\ell^2).
\]
Consequently $|\mathsf{E}_\ell(\sigma)|\le\mathsf{M}_\ell=O(\ln\ell)$ for all
$|\sigma|\le\frac12$ and all $\ell\ge1$, where
$\mathsf{M}_\ell:=\frac12\Tr(\varphi)+\frac{\ln^22}{2\pi^2}\ln\ell
+\sup_{|\sigma|\le1/2}|\mathsf{C}(\sigma)|$.
\end{lemma}

\begin{proof}
$S_\ell$ is trace class with $\lambda_n\in(0,1)$ and $\sum_n\lambda_n=\ell$,
so $\det(I-\sigma S_\ell)=\prod_n(1-\sigma\lambda_n)$ converges and no factor
vanishes for $|\sigma|<1$. Hence
\[
\ln\det(I-\sigma S_\ell)-\ell\ln(1-\sigma)=\sum_nh_\sigma(\lambda_n),
\qquad h_\sigma(t):=\ln(1-\sigma t)-t\ln(1-\sigma),
\]
and $h_\sigma(0)=h_\sigma(1)=0$, the double vanishing that removes the
term linear in $\ell$. Expanding both logarithms,
$h_\sigma(t)=\sum_{k\ge1}\frac{\sigma^k}k(t-t^k)
=\sum_{k\ge2}\frac{\sigma^k}k(t-t^k)$, the $k=1$ term cancelling identically.
For $k\ge2$ and $t\in[0,1]$ the factorization
$t-t^k=t(1-t)(1+t+\cdots+t^{k-2})$ is exact and its bracket is at most $k-1$,
so $|t-t^k|\le(k-1)t(1-t)$ and
\[
|h_\sigma(t)|\ \le\ \sum_{k\ge2}\frac{\rho^k}k(k-1)\,t(1-t)
\ \le\ t(1-t)\sum_{k\ge2}\rho^k\ =\ \frac{\rho^2}{1-\rho}\,t(1-t).
\]
Summing over $n$ and using $\sum_n\lambda_n(1-\lambda_n)=\Tr(S_\ell-S_\ell^2)$
gives the display. For the consequence take $\rho=\frac12$, so
$\rho^2/(1-\rho)=\frac12$, and bound the three pieces of
$\mathsf{P}_\ell$: $\Tr(\varphi)\le\frac1{\pi^2}\ln\ell+0.47$ by
Proposition~\ref{prop:trace} and \eqref{eq:trbounds};
$\sup_{|\sigma|\le1/2}|\ln(1-\sigma)|=\ln2$, attained at $\sigma=\frac12$;
and $\mathsf{C}$ is continuous on the compact disc $|\sigma|\le\frac12$,
hence bounded, since $\mathsf{G}(1+z)$ vanishes only at $z=-1,-2,\dots$ while
$\sup_{|\sigma|\le1/2}|\mathsf{z}(\sigma)|=\ln2/(2\pi)=0.11031\ldots<1$.
\end{proof}

The point of Lemma~\ref{lem:crude} is that it imports nothing: it uses
$\Tr S_\ell=\ell$ and Proposition~\ref{prop:trace}, which is the $m=1$ case
of Proposition~\ref{lem:higher}, proved in Section~\ref{sec:tracestep}
from the escape identity independently of everything here.
\emph{The $m=1$ case bootstraps the extraction for every $k$}, and hence,
through Lemma~\ref{lem:harmonic}, every fixed $m$.

\begin{lemma}[Taylor extraction]\label{lem:extract}
Suppose \textup{(i)} $|\mathsf{E}_\ell(\sigma)|\le\delta_\ell$ for all real
$\sigma\in[-\frac14,\frac14]$ and all $\ell\ge\ell_1$, with $\delta_\ell\to0$;
and \textup{(ii)} $|\mathsf{E}_\ell(\sigma)|\le\mathsf{M}_\ell$ for all
complex $|\sigma|\le\frac12$ and all $\ell\ge\ell_1$. Then for each fixed
$k\ge1$,
\begin{equation}\label{eq:star}
\Tr(S_\ell^{\,k})=\ell-\frac{\mathsf{H}_{k-1}}{\pi^2}\ln\ell+\mathsf{c}_k
+\mathsf{R}_k(\ell),
\qquad
|\mathsf{R}_k(\ell)|\le k\bigl(\tfrac{16}3\bigr)^k
\delta_\ell^{\,\vartheta}\mathsf{M}_\ell^{\,1-\vartheta},
\end{equation}
where $\mathsf{c}_k:=-k\,[\sigma^k]\mathsf{C}(\sigma)$ and
\[
\vartheta:=1-\frac{\ln2}{\ln(2+\sqrt3)}=0.4736\ldots
\]
\end{lemma}

\begin{proof}
\emph{Analyticity.} $\mathsf{E}_\ell$ is analytic on $|\sigma|<1$: the zeros
of $\det(I-\sigma S_\ell)$ sit at $\sigma=1/\lambda_n$, all of modulus $>1$;
the principal branch of each $\ln(1-\sigma\lambda_n)$ is valid and the sum
converges absolutely since $\sum\lambda_n=\ell<\infty$; and $\mathsf{C}$ is
analytic by the non-vanishing established in Lemma~\ref{lem:crude}.

\emph{Two constants.} The Joukowski map
$\sigma=\frac18(w+w^{-1})$ carries the annulus $1<|w|<R$ conformally onto the
interior of the ellipse $E_R$ with semi-axes
$a_R=\frac18(R+R^{-1})$ and $b_R=\frac18(R-R^{-1})$, minus the slit
$[-\frac14,\frac14]$. Two exact evaluations: $a_R=\frac12$ exactly at
$R_2=2+\sqrt3$ (solve $R+R^{-1}=4$), so $E_{R_2}$ is the largest such ellipse
inside the disc of hypothesis~(ii); and $b_R=\frac3{16}$ exactly at $R_0=2$,
so the circle $|\sigma|=\frac3{16}$ lies inside $E_{R_0}$. On an annulus the
harmonic measure of the inner boundary at $|w|=R_0$ is
$1-\ln R_0/\ln R_2$, which is $\vartheta$, evaluated in closed form rather
than estimated. Since $\ln|\mathsf{E}_\ell|$ is
subharmonic on the slit domain, at most $\ln\delta_\ell$ on the slit by~(i)
and at most $\ln\mathsf{M}_\ell$ on $\partial E_{R_2}$ by~(ii), the
two-constants theorem gives
$|\mathsf{E}_\ell|\le\delta_\ell^{\,\vartheta}\mathsf{M}_\ell^{\,1-\vartheta}$
on $E_{R_0}$, in particular on $|\sigma|=\frac3{16}$.

\emph{Cauchy.} Writing $\mathsf{E}_\ell(\sigma)=\sum_ke_k(\ell)\sigma^k$,
Cauchy's estimate on $|\sigma|=\frac3{16}$ gives
$|e_k(\ell)|\le(\frac{16}3)^k\delta_\ell^{\,\vartheta}
\mathsf{M}_\ell^{\,1-\vartheta}$.

\emph{Matching.} $\ln\det(I-\sigma S_\ell)=-\sum_{k\ge1}\sigma^k
\Tr(S_\ell^{\,k})/k$ and $[\sigma^k]\ln^2(1-\sigma)=2\mathsf{H}_{k-1}/k$, so
comparing coefficients in
$\mathsf{E}_\ell=\ln\det(I-\sigma S_\ell)-\mathsf{P}_\ell$ gives
$-k\,e_k(\ell)=\Tr(S_\ell^{\,k})
-[\ell-\frac{\mathsf{H}_{k-1}}{\pi^2}\ln\ell+\mathsf{c}_k]$, which is
\eqref{eq:star} with $\mathsf{R}_k=-k\,e_k$.
\end{proof}

Lemma~\ref{lem:crude} is hypothesis~(ii), with $\mathsf{M}_\ell=O(\ln\ell)$;
Theorem~\ref{thm:basorwidom} is hypothesis~(i), with
$\delta_\ell=O(\ln\ell/\ell)$. Substituting,
\begin{equation}\label{eq:starsub}
\Tr(S_\ell^{\,k})=\ell-\frac{\mathsf{H}_{k-1}}{\pi^2}\ln\ell+\mathsf{c}_k
+O\bigl(\ell^{-\vartheta+o(1)}\bigr),\qquad k\ge1 .
\end{equation}
The guaranteed rate is not tight: the two-constants step trades rate for
reach, converting a bound on a segment into a bound on a circle. Nothing that
follows needs more than $o(1)$.

The third lemma is pure algebra, with no analytic input at all. It is what
converts the extraction \eqref{eq:star}, whose logarithmic coefficients
are harmonic numbers, into the Beta values $\beta_m=B(m,m)$ that
Proposition~\ref{lem:higher} asserts, and it is the reason that proposition
is a theorem rather than a hypothesis.

\begin{lemma}[the alternating harmonic sum]\label{lem:harmonic}
For every integer $m\ge1$, with $\mathsf H_0=0$,
\begin{equation}\label{eq:harmid}
\sum_{j=0}^m(-1)^j\binom mj\mathsf H_{m+j-1}
\ =\ -B(m,m)\ =\ -\frac{((m-1)!)^2}{(2m-1)!}\,.
\end{equation}
\end{lemma}

\begin{proof}
Use the integral representation
\[
\mathsf H_n=\int_0^1\frac{1-t^n}{1-t}\,dt ,
\]
valid for every integer $n\ge0$: the integrand is the polynomial
$1+t+\cdots+t^{n-1}$ for $n\ge1$ and is identically $0$ for $n=0$, so no
singularity is present and the integral is elementary. The sum in
\eqref{eq:harmid} is finite, so it may be taken under the integral sign:
\[
\sum_{j=0}^m(-1)^j\binom mj\mathsf H_{m+j-1}
=\int_0^1\frac{\sum_{j=0}^m(-1)^j\binom mj\bigl(1-t^{\,m+j-1}\bigr)}{1-t}\,dt .
\]
Because $m\ge1$ we have $\sum_{j=0}^m(-1)^j\binom mj=(1-1)^m=0$, so the
constant part of the numerator vanishes identically, and the binomial theorem
collapses what remains:
\[
\begin{aligned}
\sum_{j=0}^m(-1)^j\binom mj\bigl(1-t^{\,m+j-1}\bigr)
&=-\sum_{j=0}^m(-1)^j\binom mj t^{\,m+j-1}\\
&=-t^{\,m-1}\sum_{j=0}^m\binom mj(-t)^j
\ =\ -t^{\,m-1}(1-t)^m .
\end{aligned}
\]
The numerator is thus divisible by $1-t$ as a polynomial, so the division is
exact and cancels one factor with no limiting argument:
\[
\sum_{j=0}^m(-1)^j\binom mj\mathsf H_{m+j-1}
=-\int_0^1t^{\,m-1}(1-t)^{\,m-1}\,dt=-B(m,m),
\]
the integral being finite for every $m\ge1$. Finally
$B(m,m)=\Gamma(m)^2/\Gamma(2m)=((m-1)!)^2/(2m-1)!$.
\end{proof}

\noindent
The first three values are $-1$, $-\frac16$, $-\frac1{30}$, matching
$\beta_1=1$, $\beta_2=\frac16$, $\beta_3=\frac1{30}$; the $m=2$ and $m=3$
instances are the two alternating sums evaluated directly in the proof of
Lemma~\ref{lem:exacttrace} below, which \eqref{eq:harmid} therefore
subsumes.

\subsection{The closed forms}\label{app:closed}

\begin{lemma}[exact constant terms at $m=2,3$]\label{lem:exacttrace}
Assume Theorem~\ref{thm:basorwidom}. Then
\[
\Tr(\varphi^2)=\frac1{6\pi^2}\ln\ell+\mathsf{K}_2
+O\bigl(\ell^{-\vartheta+o(1)}\bigr),
\qquad
\Tr(\varphi^3)=\frac1{30\pi^2}\ln\ell+\mathsf{K}_3
+O\bigl(\ell^{-\vartheta+o(1)}\bigr),
\]
with
\[
\mathsf{K}_2=\frac{\mathsf{A}}{6\pi^2}+\frac{\zeta(3)}{4\pi^4},
\qquad
\mathsf{K}_3=\frac{\mathsf{A}}{30\pi^2}+\frac{\zeta(3)}{16\pi^4}
+\frac{\zeta(5)}{16\pi^6},
\]
both positive. The leading coefficients agree with
Proposition~\ref{lem:higher}: $\frac1{6\pi^2}=\frac{\beta_2}{\pi^2}$ and
$\frac1{30\pi^2}=\frac{\beta_3}{\pi^2}$.
\end{lemma}

\begin{proof}
From the Barnes expansion
$\ln\mathsf{G}(1+z)=\frac z2\ln2\pi-\frac{z(z+1)}2-\frac{\gamma_{\mathrm E}z^2}2
+\sum_{n\ge3}(-1)^{n-1}\frac{\zeta(n-1)z^n}n$, the $\ln2\pi$ terms and every
odd-$n$ term cancel between $+i\mathsf{z}$ and $-i\mathsf{z}$, leaving
$\ln\mathsf{G}(1+i\mathsf{z})+\ln\mathsf{G}(1-i\mathsf{z})
=(1+\gamma_{\mathrm E})\mathsf{z}^2
+\sum_{j\ge2}(-1)^{j-1}\zeta(2j-1)\mathsf{z}^{2j}/j$. The remaining piece of
$\mathsf{C}$ is $\frac{\ln^2(1-\sigma)}{2\pi^2}\ln2\pi
=2\mathsf{z}^2\ln2\pi$, since $\ln(1-\sigma)=-2\pi\mathsf{z}$; it merges with
the $1+\gamma_{\mathrm E}$, which is where $\mathsf{A}$ comes from. Hence
\[
\mathsf{C}(\sigma)=2\mathsf{A}\mathsf{z}^2-\zeta(3)\mathsf{z}^4
+\tfrac23\zeta(5)\mathsf{z}^6-\tfrac12\zeta(7)\mathsf{z}^8+\cdots
\]
Since $\mathsf{z}(\sigma)=O(\sigma)$, the $\mathsf{z}^8$ term contributes
nothing below $\sigma^8$: for $k\le6$ the $\zeta(7)$ term and beyond are
absent, not neglected. Extracting $\mathsf{c}_k=-k[\sigma^k]\mathsf{C}$,
\[
\begin{aligned}
\mathsf{c}_2&=-\frac{\mathsf{A}}{\pi^2}, &
\mathsf{c}_3&=-\frac{3\mathsf{A}}{2\pi^2}, &
\mathsf{c}_4&=-\frac{11\mathsf{A}}{6\pi^2}+\frac{\zeta(3)}{4\pi^4},\\
\mathsf{c}_5&=-\frac{25\mathsf{A}}{12\pi^2}+\frac{5\zeta(3)}{8\pi^4}, &
\mathsf{c}_6&=-\frac{137\mathsf{A}}{60\pi^2}+\frac{17\zeta(3)}{16\pi^4}
-\frac{\zeta(5)}{16\pi^6}. & &
\end{aligned}
\]
Now $\varphi^m=S_\ell^{\,m}(I-S_\ell)^m$ gives the exact finite expansions
$\Tr(\varphi^2)=\Tr(S_\ell^2)-2\Tr(S_\ell^3)+\Tr(S_\ell^4)$ and
$\Tr(\varphi^3)=\Tr(S_\ell^3)-3\Tr(S_\ell^4)+3\Tr(S_\ell^5)-\Tr(S_\ell^6)$.
The coefficients sum to zero in both, so the terms linear in $\ell$ cancel.
For the logarithms,
$-\frac1{\pi^2}(\mathsf{H}_1-2\mathsf{H}_2+\mathsf{H}_3)
=-\frac1{\pi^2}\cdot(-\frac16)=\frac1{6\pi^2}$ and
$-\frac1{\pi^2}(\mathsf{H}_2-3\mathsf{H}_3+3\mathsf{H}_4-\mathsf{H}_5)
=-\frac1{\pi^2}\cdot(-\frac1{30})=\frac1{30\pi^2}$. For the constants,
$\mathsf{c}_2-2\mathsf{c}_3+\mathsf{c}_4=\mathsf{K}_2$ and
$\mathsf{c}_3-3\mathsf{c}_4+3\mathsf{c}_5-\mathsf{c}_6=\mathsf{K}_3$, by the
same arithmetic on the coefficients $\mathsf{c}_k$ above. Each remainder is a fixed finite
combination of the $\mathsf{R}_k$ of \eqref{eq:starsub}, $k\le6$.
\end{proof}

\begin{corollary}[the moment combination, with its constant]\label{cor:Bconst}
Assume Theorem~\ref{thm:basorwidom}. With
$\mathsf{B}:=256\mathsf{K}_3-48\mathsf{K}_2$,
\[
256\Tr(\varphi^3)-48\Tr(\varphi^2)
=\frac8{15\pi^2}\ln\ell+\mathsf{B}+O\bigl(\ell^{-\vartheta+o(1)}\bigr),
\]
\[
\mathsf{B}=\frac{8\mathsf{A}}{15\pi^2}+\frac{4\zeta(3)}{\pi^4}
+\frac{16\zeta(5)}{\pi^6}\ >\ \frac{27}{125}\ >\ 0 .
\]
\end{corollary}

\begin{proof}
The leading coefficient is
$256\beta_3-48\beta_2=\frac{256}{30}-\frac{48}6=\frac{128}{15}-8=\frac8{15}$
by Lemma~\ref{lem:exacttrace}, and the constant is
$256\mathsf{K}_3-48\mathsf{K}_2$ by the same lemma; collecting
$\mathsf{A}$, $\zeta(3)$ and $\zeta(5)$ gives the displayed closed form.
Positivity follows from exact rational estimates. The
three coefficients are strictly positive. Further
$\gamma_{\mathrm E}>\frac12$ and $\ln2\pi>\frac32$ (because
$2\pi>6>e^{3/2}=4.4816\ldots$), so $\mathsf{A}>3$; $\zeta(3),\zeta(5)>1$; and
$\pi<\frac{63}{20}$, whence $\pi^2<10$, $\pi^4<100$, $\pi^6<1000$. Therefore
\[
\mathsf{B}\ >\ \frac{8\cdot3}{15\cdot10}+\frac4{100}+\frac{16}{1000}
\ =\ \frac4{25}+\frac1{25}+\frac2{125}\ =\ \frac{27}{125}.
\]
\end{proof}

\begin{corollary}[window count at the full $c_0$, for every
$\ell>0$]\label{cor:windowexact}
Assume Theorem~\ref{thm:basorwidom}. There is $\ell_0\ge1$ such that, with
$C_0:=c_0\ln\ell_0$ and $c_0=\frac8{15\pi^2}$,
\[
M_a=\#\{n:\lambda_n(\ell)\in(\tfrac14,\tfrac34)\}\ \ge\ c_0\ln\ell-C_0
\qquad\text{for every }\ell>0,
\]
and moreover $M_a\ge c_0\ln\ell+\mathsf{B}/2$ for every $\ell\ge\ell_0$.
\end{corollary}

\begin{proof}
Proposition~\ref{prop:count} gives $M_a\ge256\Tr(\varphi^3)-48\Tr(\varphi^2)$
unconditionally, and Corollary~\ref{cor:Bconst} writes the right side as
$c_0\ln\ell+\mathsf{B}+\mathsf{R}(\ell)$ with $\mathsf{R}(\ell)\to0$. Choose
$\ell_0\ge1$ beyond which $|\mathsf{R}(\ell)|\le\mathsf{B}/2$ (possible
because $\mathsf{B}>0$ and $\mathsf{R}\to0$), and put
$C_0:=c_0\ln\ell_0\ge0$. For $\ell\ge\ell_0$,
$M_a\ge c_0\ln\ell+\mathsf{B}-\mathsf{B}/2\ge c_0\ln\ell\ge c_0\ln\ell-C_0$.
For $0<\ell<\ell_0$, $c_0\ln\ell-C_0=c_0\ln(\ell/\ell_0)<0\le M_a$, since
$M_a$ is a nonnegative integer. The two branches cover $(0,\infty)$.
\end{proof}

\noindent
The coefficient survives the passage to all $\ell$, and that is the point. A
route that had to shave a positive amount off $c_0$ to buy room in the
remainder would deliver less than $c_0$; here $\mathsf{B}>0$ pays for that
room instead, so $c_0=\frac8{15\pi^2}$ passes through intact. This says
something about the passage from the asymptotic to the all-$\ell$ form and
nothing more. In particular it does not say that $c_0$ is the best constant
available for the window $(\frac14,\frac34)$; Remark~\ref{rem:c0status}
records that it is not.

The threshold $\ell_0$, on the other hand, is not pinned down. It depends on
the implied constant in the $O(\ln s/s)$ of Theorem~\ref{thm:basorwidom},
which \cite{Charlier21} states to be uniform but does not quantify. So $c_0$
and $\mathsf{B}$ are exact and explicit, while $C_0$ exists and is finite but
is \emph{not effective}. Theorem~\ref{thm:window} is worded accordingly.

\bibliographystyle{plain}
\bibliography{bibliography}

\end{document}